\documentclass[oneside]{amsart}
\usepackage{amssymb}
\usepackage{amsmath}
\usepackage{amsfonts}
\usepackage{amsthm,thmtools}
\usepackage{hyperref}
\usepackage{amscd}
\usepackage{color}
\usepackage{enumerate}
\usepackage{xypic}
\usepackage{tikz}
\usepackage{tikz-cd}
\usepackage{caption}
\usepackage{subcaption}
\usepackage{enumitem}
\usepackage[foot]{amsaddr}

\usepackage{fullpage}
\usetikzlibrary{shapes.misc}

\tikzset{cross/.style={cross out, draw=black, minimum size=2*(#1-\pgflinewidth), inner sep=0pt, outer sep=0pt},
cross/.default={1pt}}

\newcommand{\ZZ}{\mathbb{Z}}
\newcommand{\KK}{\mathbb{K}}
\newcommand{\RR}{\mathbb{R}}
\newcommand{\PP}{\mathbb{P}}

\newcommand{\NN}{\mathbb{N}}

\newtheorem{mainthm}{Main Theorem}
\newtheorem{thm}{Theorem}[section]
\newtheorem{prop}[thm]{Proposition}
\newtheorem{lemma}[thm]{Lemma}
\newtheorem{cor}[thm]{Corollary}

\theoremstyle{definition}
\declaretheorem[name=Example,qed={\lower-0.3ex\hbox{$\triangleleft$}},sibling=thm]{ex}

\newtheorem{example}[thm]{Example}
\newtheorem{conj}[thm]{Conjecture}
\newtheorem{step}{Step}
\newtheorem{case}{Case}

\theoremstyle{definition}

\newtheorem{defn}[thm]{Definition}
\newtheorem{rem}[thm]{Remark}

\title{Classifying Fano Complexity-One $T$-Varieties via Divisorial Polytopes}

\author{Nathan Ilten, Marni Mishna, and Charlotte Trainor}
\address{Department of Mathematics, Simon Fraser University,
8888 University Drive, Burnaby BC V5A1S6, Canada}
\email{nilten@sfu.ca, mmishna@sfu.ca, ctrainor@sfu.ca}

\begin{document}

\begin{abstract}
The correspondence between Gorenstein Fano toric varieties and reflexive polytopes has been generalized by Ilten and S\"u\ss{} to a correspondence between Gorenstein Fano complexity-one $T$-varieties and Fano divisorial polytopes. Motivated by the finiteness of reflexive polytopes in fixed dimension, we show that over a fixed base polytope, there are only finitely many Fano divisorial polytopes, up to equivalence. We classify two-dimensional Fano divisorial polytopes, recovering Huggenberger's classification of Gorenstein del Pezzo $\KK^*$-surfaces. Furthermore, we show that any three-dimensional Fano divisorial polytope is equivalent to one involving only eight functions.
\end{abstract}
\maketitle

\section{Introduction}
It is well-known that Fano toric varieties with at worst Gorenstein
singularities correspond to so-called reflexive polytopes, see
e.g. \cite[\S8.3]{CLS}. Furthermore, in any fixed dimension, there are
only a finite number of reflexive polytopes up to equivalence
\cite{lagarias:91a}, and there is an algorithm for classifying them
\cite{kreuzer:97a}. It follows that there are only finitely many Gorenstein Fano
 toric varieties in a fixed dimension, and they may be
classified algorithmically.

A natural generalization of toric varieties are \emph{complexity-one $T$-varieties}: normal varieties $X$ equipped with the effective action of an algebraic torus $T$ whose generic orbit has codimension one in $X$. Any toric variety $X$ may be considered as a complexity-one $T$-variety by restricting the action of the big torus on $X$ to a codimension-one subtorus $T$.

In a manner similar to the case of toric varieties, these varieties can also be
encoded using quasi-combinatorial data, specifically, a generalization
of polytopes. Throughout, suppose~$M$ is a lattice. 

\begin{defn}
\label{CDP}
A \emph{combinatorial divisorial polytope} (CDP) with respect to the
lattice $M$ consists of a full-dimensional lattice polytope
$\Box\subset M\otimes \RR$, along with an $n$-tuple $\Psi=(\Psi_1,\ldots,\Psi_n)$ (for some $n\in \NN$)
 of piecewise-affine concave functions
$\Psi_i:\Box\to\RR$ such that
	\begin{enumerate}
		\item For each $i$, the graph of $\Psi_i$ is a polyhedral complex with integral vertices;
		\item For each $u\in \Box^\circ$, $\sum \Psi_i(u) >0$.
	\end{enumerate}
	We call $\Box$ the \emph{base} of $\Psi$.
	The \emph{dimension} of $\Psi$ is $\dim \Box+1$.
\end{defn}
\noindent Attaching each of the functions $\Psi_i$ to a point $P_i$ in the curve $\PP^1$ gives rise to a \emph{divisorial polytope}  on $\PP^1$; these correspond to rational polarized complexity-one $T$-varieties \cite{polarized}.

Just as there is a special subclass of lattice polytopes corresponding
to Gorenstein Fano toric varieties, there is a special subclass of
CDPs corresponding to Fano complexity-one $T$-varieties with at worst
canonical Gorenstein singularities \cite[Definition 3.3 and Theorem
3.5]{stability}. We call such CDPs \emph{Fano}, and recall the details
in Definition~\ref{Fano}.

Considering the finiteness result for reflexive polytopes, one might
ask if a similar result holds for Fano CDPs. To pose this question we
define a natural notion of \emph{equivalence} of CDPs in
\S~\ref{sec:equiv}. We say that a CDP is \emph{toric} if it is
equivalent to a CDP consisting of at most two functions
$\Psi_1,\Psi_2$. Our conjecture for Fano CDPs is the following:

\begin{conj}\label{conj:1}
In any fixed dimension $d$, there are only finitely many equivalence classes of non-toric Fano CDPs.
\end{conj}
\noindent Geometrically, this conjecture is equivalent to stating that
there are only finitely many families of canonical Gorenstein Fano
complexity-one $T$-varieties in any given dimension.
It should be noted that any Fano toric variety can be considered as a complexity-one $T$-variety in infinitely many ways, which is why we exclude the toric case from the above conjecture.

\subsection{Main results}
Our first main result is the following:
\begin{mainthm}[Theorem \ref{fixedBase}]\label{thm:2}
Over any fixed base $\Box$, there are at most finitely many equivalence classes of Fano CDPs.
\end{mainthm}
\noindent While we do not present it as such, this result can be made
effective. The proof of Main
Theorem~\ref{thm:2} suggests an algorithm to enumerate all equivalence
classes of Fano CDPs over a fixed base.

Given Main Theorem~\ref{thm:2}, to prove Conjecture~\ref{conj:1} it is
sufficient to prove that only finitely many base
polytopes $\Box$ occur for non-toric Fano CDPs in any given
dimension. Equivalently, a bound on the volume of the bases which
occur (or the number of interior lattice points \cite{lagarias:91a}) would also prove the conjecture.

We then apply the tools we develop in the proof of Main Theorem \ref{thm:2} to analyze low-dimensional cases:
\begin{mainthm}[Theorem~\ref{thm:2d}]\label{thm:3}
There are exactly 34 equivalence classes of two-dimensional non-toric Fano CDPs.
\end{mainthm}

\begin{mainthm}[Theorem~\ref{thm:3dbound}]\label{thm:4}
Any three-dimensional Fano CDP is equivalent to one with at most eight functions.
\end{mainthm}
\noindent Our analysis in these cases leads us to conjecture (Conjecture \ref{conj:bound}) that any $d$-dimensional Fano CDP is equivalent to one with at most $2^d$ functions.

\subsection{Geometric interpretations of the main results}
Our results have an interesting geometric interpretation.
Given a  Fano complexity-one $T$-variety $X$ with at worst canonical Gorenstein singularities, let $\pi:\widetilde{X}\to \PP^1$ be the resolution of the rational quotient map $X\dashrightarrow \PP^1$.  

\begin{cor}\label{cor:1}
		For a fixed polarized toric variety $Y$, there are finitely many families of canonical Gorenstein Fano complexity-one $T$-varieties $X$ such that the general fiber of $\pi:\widetilde{X}\to \PP^1$ is isomorphic to $Y$, polarized with regards to the pullback of $-K_X$. 
\end{cor}

\begin{cor}\label{cor:2}
There are exactly 34 families of non-toric Gorenstein del Pezzo $\KK^*$-surfaces $X$.
\end{cor}

\begin{cor}\label{cor:3}
	For any canonical Gorenstein Fano complexity-one threefold, the resolved quotient map $\pi$ has at most $8$ non-integral fibers.
\end{cor}

\noindent These corollaries are direct translations of Main Theorems \ref{thm:2}, \ref{thm:3}, and \ref{thm:4}, respectively.

An alternate approach to the classification of Fano complexity-one $T$-varieties is via a classification of their Cox rings. This has been employed by Huggenberger \cite{huggenberger,hausen:13a} to classify Gorenstein (log) del Pezzo surfaces with a $\KK^*$-action; our Main Theorem \ref{thm:3} and Corollary \ref{cor:2} recover this classification in the del Pezzo case. A translation of this classification into the language of divisorial polytopes used here is found in \cite{hendrik}.

This approach via Cox rings has also been employed in \cite{hausen:11a,bechtold:16a,pic2} to classify higher-dimensional Fano complexity-one $T$-varieties of Picard ranks one and two. These techniques should work more generally, as long as one has a bound on the Picard rank of the varieties in question.
Interestingly, while our Main Theorem \ref{thm:4} does not bound the Picard rank of Fano complexity-one threefolds, it does limit the number of equations which may appear in the defining ideal of the Cox ring. We hope that by combining this bound with the Cox ring approach, one may obtain a complete classification of canonical Gorenstein Fano complexity-one threefolds.

\subsection{Outline of the article}
First we provide some background information on CDPs in
\S\ref{sec:FCDPs}, including a description of operations preserving
equivalence, as well as establishing the relationship between lattice
polytopes and toric CDPs. We recall the definition of a Fano CDP
given by Ilten and S\"u\ss{} in \cite{stability}, and use the notion
of equivalence to make some key simplifying assumptions about the
properties of Fano CDPs.  In \S\ref{sec:Bounds} we provide, for fixed base polytope $\Box$, an
upper bound on the number of functions in a Fano CDP. Main Theorem
\ref{thm:2}, stating the finiteness of the number of equivalence
classes of Fano CDPs over any fixed base, is established in
\S\ref{sec:Base}. In \S\ref{sec:surfaces}, we use the tools
developed throughout this paper to provide a classification of all
equivalence classes of two-dimensional Fano CDPs, proving Main Theorem \ref{thm:3}. Finally, in \S\ref{sec:threefolds}, we prove Main Theorem \ref{thm:4}.

\section{Generalizing reflexive polytopes to Fano CDPs}
\label{sec:FCDPs}
In this section, we introduce the notion of equivalence of CDPs, discuss how CDPs can be viewed as a generalization of lattice polytopes, and introduce the Fano property of CDPs.

\subsection{Equivalence of CDPs}\label{sec:equiv}
\begin{table}
\centering
\begin{tabular}{ccccccc}
\begin{tikzpicture}[scale=0.8]
\draw[gray, very thin] (-1,-1) grid (1,2);
\filldraw[gray] (0,0) circle (2pt);
\filldraw[gray] (-1,0) circle (2pt);
\filldraw[gray] (1,0) circle (2pt);
\draw[gray, line width=1] (-1,0) -- (1,0);
\draw[black, line width=1] (-1,0)--(0,1)--(1,0);
\draw[black, line width=1] (-1,0)--(1,1);
\end{tikzpicture}
&
\begin{tikzpicture}[baseline=0ex,scale=0.8]
\draw[->,line width=1] (-0.5,1.5)--(0.5,1.5);
\end{tikzpicture}
& 
\begin{tikzpicture}[scale=0.8]
\draw[gray, very thin] (-1,-1) grid (1,2);
\filldraw[gray] (0,0) circle (2pt);
\filldraw[gray] (-1,0) circle (2pt);
\filldraw[gray] (1,0) circle (2pt);
\draw[gray, line width=1] (-1,0) -- (1,0);
\draw[black, line width=1] (-1,-1)--(0,1)--(1,1);
\draw[black, line width=1] (-1,1)--(1,0);
\end{tikzpicture}
&
\begin{tikzpicture}[baseline=0ex,scale=0.8]
\draw[->,line width=1] (-0.5,1.5)--(0.5,1.5);
\end{tikzpicture}
& 
\begin{tikzpicture}[scale=0.8]
\draw[gray, very thin] (-1,-1) grid (1,2);
\filldraw[gray] (0,0) circle (2pt);
\filldraw[gray] (-1,0) circle (2pt);
\filldraw[gray] (1,0) circle (2pt);
\draw[gray, line width=1] (-1,0) -- (1,0);
\draw[black, line width=1] (-1,0)--(0,2)--(1,2);
\draw[black, line width=1] (-1,0)--(1,-1);
\end{tikzpicture}
&
\begin{tikzpicture}[baseline=0ex,scale=0.8]
\draw[->,line width=1] (-0.5,1.5)--(0.5,1.5);
\end{tikzpicture}
& 
\begin{tikzpicture}[scale=0.8]
\draw[gray, very thin] (-1,-1) grid (1,2);
\filldraw[gray] (0,0) circle (2pt);
\filldraw[gray] (-1,0) circle (2pt);
\filldraw[gray] (1,0) circle (2pt);
\draw[gray, line width=1] (-1,0) -- (1,0);
\draw[black, line width=1] (-1,2)--(0,2)--(1,0);
\draw[black, line width=1] (-1,-1)--(1,0);
\end{tikzpicture} 
\end{tabular}

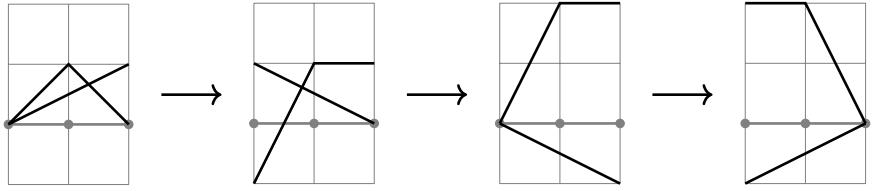
\captionof{figure}{Example of equivalent CDPs. The base polytope of the leftmost CDP is $\Box=[-1,1]$. First we shear one function by a factor of $-1$ and the other by a factor of $1$; next we translate the two functions; finally we transform the base by reflecting through the origin.}
\label{fig:eq}
\end{table}


Let $\Psi$ be a $d$-dimensional CDP with base $\Box$ and functions $\Psi_1,\dots,\Psi_n$. We may perform any combination of the following actions or their inverses to obtain an \emph{equivalent} CDP:
\begin{description}[itemsep=1em]
\item[Addition of the zero function] If $0\equiv \Psi_{n+1}:\Box\to
  \RR$, then $(\Psi_1,\ldots,\Psi_n,\Psi_{n+1})$ is equivalent to
  $\Psi$. 

\item[Permutation/Relabelling] For $\sigma\in S_n$, the CDP with base $\Box$ and
  functions \[(\Psi_{\sigma(1)},\ldots,\Psi_{\sigma(n)})\] is
  equivalent to $\Psi$. 

\item[Transformation of the base] If $\phi$ is an invertible
  affine linear transformation of the lattice $M$, the CDP with base $\phi(\Box)$ and functions $\Psi_i\circ\phi^{-1}$
is equivalent to~$\Psi$. 

\item[Translation] For $\alpha_1,\dots,\alpha_n\in\mathbb{Z}$ with $\sum_{i=1}^n\alpha_i=0$, the CDP with base $\Box$ and functions
$\Psi_i+\alpha_i$
is equivalent to $\Psi$. 

\item[Shearing Action] For $v\in M^*$ and $\beta_1,\dots,\beta_n\in\mathbb{Z}$ with $\sum_{i=1}^n\beta_i=0$, the CDP with base $\Box$ and functions 
\[ u\mapsto \Psi_i(u)+\beta_i\cdot\langle u,v\rangle \quad \text{for all } u\in\Box, \] is equivalent to $\Psi$. 
\end{description}


\noindent In Figure~\ref{fig:eq}, we illustrate the equivalence operations of shearing, translating, and transforming the base.

\begin{rem}
The equivalences are motivated geometrically as follows. Given a CDP~$\Psi$, attaching points $P_i\in \PP^1$ to each function $\Psi_i$ gives rise to a \emph{divisorial polytope} on $\PP^1$; these correspond to rational polarized complexity-one $T$-varieties \cite{polarized}. Any two equivalent CDPs will give rise to complexity-one $T$-varieties which are equivariantly isomorphic, after appropriate choice of the points $P_i$. 
\end{rem}

\subsection{From polytopes to toric CDPs}
Consider a polytope $P$ in $(M\times \ZZ)\otimes \RR$, with vertices in the lattice $M\times \ZZ$. This gives rise to a CDP with two functions as follows.
Let $\pi_1$ be the projection to $M\otimes \RR$ and $\pi_2$ be the projection to $\RR$. Set $\Box=\pi_1(P)$ and define $\Psi_1,\Psi_2:\Box\rightarrow \RR$ by \[ \Psi_1(u)=\max(\pi_2(\pi_1^{-1}(u)\cap P)) \quad \text{ and } \quad \Psi_2(u)=-\min(\pi_2(\pi_1^{-1}(u)\cap P)). \]
Then the base $\Box$ with the functions $\Psi_1$, $\Psi_2$ is a CDP. This process is illustrated with an example in Figure~\ref{corex}. Conversely, this process can be inverted to obtain a lattice polytope from a CDP with two functions by reflecting one of the functions and taking the convex hull of the vertices of the graph of the CDP. 

\begin{figure}
\centering
\hspace*{\fill}
\begin{subfigure}[b]{0.25\columnwidth}
\resizebox{\textwidth}{!}{
\begin{tikzpicture}
\draw[step=1cm,gray,very thin] (-8,-2) grid (-3,2);
\draw[black, line width=2] (-8,0)--(-8,1)--(-5,2)--(-4,1)--(-3,-1)--(-6,-2)--(-8,0);
\draw[black, very thin] (-5,2)--(-5,-2);
\draw[black, very thin] (-8,0)--(-3,0);
\filldraw[black] (-5,0) circle (2pt);
\end{tikzpicture}
}
\subcaption{Polytope $P$}
\label{P}
\end{subfigure}
\hfill
\begin{subfigure}[b]{0.25\columnwidth}
\resizebox{\textwidth}{!}{
\begin{tikzpicture}
\draw[gray, very thin] (-3,-2) grid (2,2);
\filldraw[gray] (-3,0) circle (2pt);
\filldraw[gray] (-2,0) circle (2pt);
\filldraw[gray] (-1,0) circle (2pt);
\filldraw[gray] (1,0) circle (2pt);
\filldraw[gray] (2,0) circle (2pt);
\draw[gray, line width=2] (-3,0)--(2,0);
\filldraw[black] (0,0) circle (2pt);
\draw[line width=2] (-3,1)--(0,2)--(1,1)--(2,-1);
\draw[dotted,line width=2] (-3,0)--(-1,-2)--(2,-1);
\end{tikzpicture}
}
\subcaption{$\Box=\pi_1(P)$}
\label{R}
\end{subfigure}
\hfill
\begin{subfigure}[b]{0.25\columnwidth}
\resizebox{\textwidth}{!}{
\begin{tikzpicture}
\draw[gray, very thin] (-3,-2) grid (2,2);
\filldraw[gray] (-3,0) circle (2pt);
\filldraw[gray] (-2,0) circle (2pt);
\filldraw[gray] (-1,0) circle (2pt);
\filldraw[gray] (1,0) circle (2pt);
\filldraw[gray] (2,0) circle (2pt);
\draw[gray, line width=2] (-3,0)--(2,0);
\filldraw[black] (0,0) circle (2pt);
\draw[line width=2] (-3,1)--(0,2)--(1,1)--(2,-1);
\draw[dotted,line width=2] (-3,0)--(-1,2)--(2,1);
\end{tikzpicture}
}
\subcaption{Resulting CDP}
\label{fctns}
\end{subfigure}
\hspace*{\fill}
\caption{Example of correspondence between polytopes and CDPs with two functions: {\sc(a)} the lattice polytope; {\sc(b)} the lattice polytope projected down one dimension; and {\sc(c)} the resulting CDP.}
\label{corex}
\end{figure}
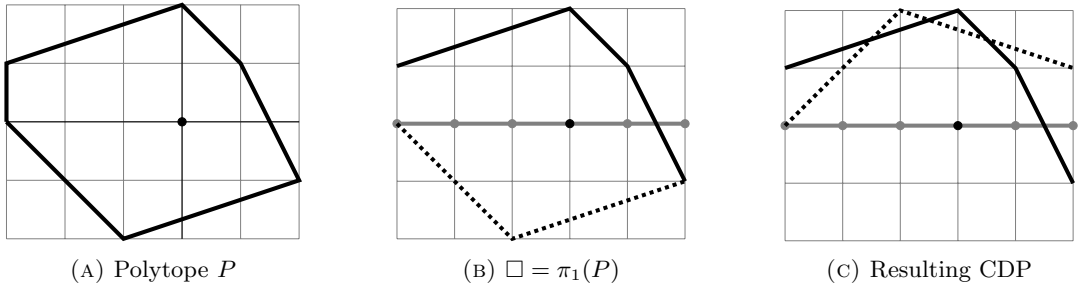

A CDP  $\Psi$ is \emph{toric} if it is equivalent to a CDP with at most two functions. Since we may add a constant zero function to obtain an equivalent CDP, we have that toric CDPs are those equivalent to a CDP with exactly two functions. Thus, toric CDPs are exactly those which arise from a lattice polytope via the above construction.

\begin{rem}
\label{rem:equivcor}
If two toric CDPs  are equivalent, then their corresponding polytopes are equivalent as lattice polytopes.
\end{rem}

The converse of Remark~\ref{rem:equivcor} does not hold; that is, equivalent polytopes do not necessarily correspond to equivalent CDPs. 
Examples of this fact can be found by shearing a polytope in a direction other than the direction that we project along to obtain the polytope's corresponding CDP.  Such an example is given in Figure~\ref{ineqFCDPs}; the CDPs depicted there are inequivalent because their base polytopes are inequivalent.
From the point of view of algebraic geometry, this is saying that a toric variety may be given the structure of a complexity-one $T$-variety by restricting to a codimension-one subtorus in infinitely many different ways.

\begin{table}
\centering
\begin{tabular}{cccccc}
\begin{tikzpicture}[scale=0.8]
\draw[step=1cm,gray,very thin] (-1,-1) grid (1,1);
\filldraw[black] (0,0) circle (2pt);
\draw[black, line width=1] (-1,0)--(0,1)--(1,0)--(1,-1)--(-1,0);
\end{tikzpicture}
&
\begin{tikzpicture}[baseline=0ex, scale=0.8]
\draw[->,color=black, line width=1] (-0.5,1)--(0.5,1);
\end{tikzpicture}
&
\begin{tikzpicture}[scale=0.8]
\draw[step=1cm,gray,very thin] (-1,-1) grid (2,1);
\draw[black, line width=1] (-1,0)--(-1,1)--(1,0)--(2,-1)--(-1,0);
\filldraw[black] (0,0) circle (2pt);
\end{tikzpicture}
&
\begin{tikzpicture}[baseline=0ex, scale=0.8]
\draw[->,color=black, line width=1] (-0.5,1)--(0.5,1);
\end{tikzpicture}
&
\begin{tikzpicture}[scale=0.8]
\draw[step=1cm,gray,very thin] (-2,-1) grid (3,1);
\draw[black, line width=1] (-2,1)--(1,0)--(3,-1);
\draw[black, line width=1] 
(3,-1)--(-1,0)--(-2,1);
\filldraw[black] (0,0) circle (2pt);
\end{tikzpicture}

\\

\begin{tikzpicture}[scale=0.8]
\draw[->, color=black, line width=1] (0,0.5)--(0,-0.5);
\end{tikzpicture}
& { } &
\begin{tikzpicture}[scale=0.8]
\draw[->, color=black, line width=1] (0,0.5)--(0,-0.5);
\end{tikzpicture}
& &
\begin{tikzpicture}[scale=0.8]
\draw[->, color=black, line width=1] (0,0.5)--(0,-0.5);
\end{tikzpicture}

\\

\begin{tikzpicture}[scale=0.8]
\draw[step=1cm,gray,very thin] (-1,-1) grid (1,1);
\draw[white] (0,-1) circle (1pt);
\filldraw[gray] (0,0) circle (2pt);
\filldraw[gray] (-1,0) circle (2pt);
\filldraw[gray] (1,0) circle (2pt);
\draw[gray, line width=1] (-1,0) -- (1,0);
\draw[black, line width=1] (-1,0)--(0,1)--(1,0);
\draw[black, line width=1] (-1,0)--(1,1);
\end{tikzpicture}
& { } &
\begin{tikzpicture}[scale=0.8]
\draw[step=1cm,gray,very thin] (-1,-1) grid (2,1);
\filldraw[gray] (0,0) circle (2pt);
\filldraw[gray] (-1,0) circle (2pt);
\filldraw[gray] (1,0) circle (2pt);
\filldraw[gray] (2,0) circle (2pt);
\draw[gray, line width=1] (-1,0)--(2,0);
\draw[black, line width=1] (-1,1)--(1,0)--(2,-1);
\draw[black, line width=1] (2,1)--(-1,0);
\end{tikzpicture}
& &
\begin{tikzpicture}[scale=0.8]
\draw[step=1cm,gray,very thin] (-2,-1) grid (3,1);
\filldraw[gray] (0,0) circle (2pt);
\filldraw[gray] (-1,0) circle (2pt);
\filldraw[gray] (-2,0) circle (2pt);
\filldraw[gray] (1,0) circle (2pt);
\filldraw[gray] (2,0) circle (2pt);
\filldraw[gray] (3,0) circle (2pt);
\draw[gray, line width=1] (-2,0)--(3,0);
\draw[black, line width=1] (-2,1)--(1,0)--(3,-1);
\draw[black, line width=1]
(3,1)--(-1,0)--(-2,-1);
\end{tikzpicture}

\end{tabular}

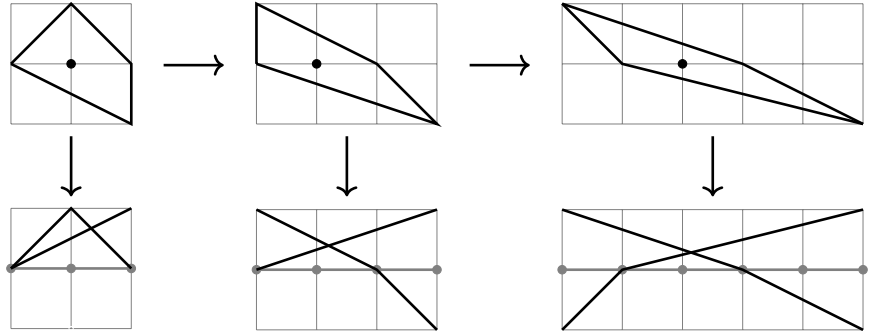
\captionof{figure}{Equivalent reflexive polytopes can correspond to inequivalent Fano CDPs.}
\label{ineqFCDPs}
\end{table}

\subsection{Fano Divisorial Polytopes}\label{sec:divpoly}
The \emph{Fano} property of a CDP, as defined by Ilten and S\"u\ss{} in \cite{stability}, generalizes the reflexive property of polytopes.
Under the correspondence between divisorial polytopes and polarized rational complexity-one $T$-varieties, it corresponds exactly to canonical Gorenstein Fano $T$-varieties with anti-canonical polarization. We recall this property below in the context of CDPs.

For a polytope $\Box$, we denote its interior by $\Box^\circ$ and its boundary by $\partial\Box$. For a function $\Psi_i$, we denote the graph of $\Psi_i$ by $\Gamma(\Psi_i)$. A facet $F$ of a lattice polytope $P$ in $(M\times \ZZ)\otimes \RR$ is \emph{in height one} if there exists some $u\in M^*\times \ZZ$ such that $\langle v,u\rangle=1$ for all $v\in F$.

\begin{defn}
\label{Fano}

A CDP is \emph{Fano} if it is equivalent to a CDP~$\Psi$ with base polytope~$\Box$ and functions $\Psi_1,\dots,\Psi_n$ for which there are integers $a_1,\dots,a_n$ such that 
\begin{enumerate}
\item $0\in\Box^\circ$;
\item $\sum_{i=1}^na_i=-2$; \label{cdpdos}
\item For all $i=1\dots n$,  $\Psi_i(0)+a_i+1>0$, and each facet of $\Gamma(\Psi_i+a_i+1)$ is at height one; \label{cdptres}
\item For any facet $F$ of $\Box$ not at height one, $\sum_{i=1}^n\Psi_i\equiv0$ on $F$.
\end{enumerate}
\end{defn}
\noindent When each facet of $\Gamma(\Psi_i+a_i+1)$ is at height one, we say that $\Psi_i+a_i+1$ is \emph{at height~one}. 

\begin{rem}
\label{rem:Fanoprop}
The four properties of Definition~\ref{Fano} are preserved for
equivalent CDPs, provided that any transformation of the base
preserves the origin.
\end{rem}

\begin{rem}
\label{rem:FanoRefl}It is straightforward to check that under the construction producing a toric CDP from a lattice polytope in $M\times \ZZ$, the CDP is Fano if and only if its corresponding polytope is isomorphic to a reflexive polytope.
\end{rem}

A consequence of Remark~\ref{rem:FanoRefl}, along with the fact that the converse of Remark~\ref{rem:equivcor} does not hold in general, is that 
infinite families of inequivalent toric Fano CDPs can be constructed from a single isomorphism class of reflexive polytopes. In fact, Figure~\ref{ineqFCDPs} gives an example of how to construct such a family. Such examples explain the restriction of the statement of Conjecture~\ref{conj:1} to non-toric CDPs.

\subsection{Normalization}\label{sec:normalization}
For the purposes of classifying Fano CDPs, it would be useful to have some kind of normal form, that is, a distinguished representative in each equivalence class of Fano CDPs. While we have yet to find a natural normal form for Fano CDPs, in the following we describe a number of normalizing assumptions we will be making whenever considering a Fano CDP. This will be key to our strategy for obtaining bounds on the structure of Fano CDPs.

\begin{defn}\label{def:normalized}
A Fano CDP $\Psi=(\Psi_1,\ldots,\Psi_n)$ is \emph{normalized} if 
\begin{enumerate}
\item $\Psi$ satisfies the four criteria of Definition \ref{CDP};
\item If $\Psi_i$ is linear, then it does not have integral slope. 
\end{enumerate}
\end{defn}
Note that any non-toric Fano CDP $\Psi$ is equivalent to a normalized Fano CDP. Indeed, by virtue of being Fano, we can certainly satisfy the first property above by replacing $\Psi$ with an equivalent CDP. Furthermore, given any linear $\Psi_i$ with integral slope, we can shear and translate so that the corresponding function becomes zero, and then eliminate it. We repeat this until we have only two functions remaining (and have a toric CDP), or there are no more linear functions with integral slope.

In the remainder of this paper, we will always be dealing with a normalized Fano CDP $\Psi$ with base polytope $\Box$ and functions
 $\Psi_1,\dots,\Psi_n$. In particular, there are integers $a_i$ satisfying properties \ref{cdpdos} and \ref{cdptres} of Definition \ref{CDP}. It is often more convenient to  
 consider the translated functions
 $\Psi_i':=\Psi_i+a_i+1$, which are at height~one. Applying Property~\ref{cdpdos}
 of Definition \ref{CDP} to the translated functions $\Psi_i'$ yields

\begin{equation}
\label{eq1}
\sum_{i=1}^n \Psi_i'  > n-2 \quad \text{ on $\Box^\circ$} 
\end{equation}
and
\begin{equation}
\label{eq2}
 \sum_{i=1}^n \Psi_i'  \geq n-2 \quad \text{on $\partial\Box$.}
\end{equation}
Furthermore, applying Property~\ref{cdptres} of Definition~\ref{Fano} to the translated functions $\Psi_i'$ yields
\begin{equation}
\label{eq3}
\sum_{i=1}^n\Psi_i' \equiv n-2
\end{equation}
on any facet of $\Box$ that is not at height one.
We almost exclusively work with the translated functions $\Psi_i'$ in order to most easily exploit the property that the facets of their graphs are at height one. 

\section{Bounds on Number of Functions}
\label{sec:Bounds}
\subsection{Overview}
In this section we establish a bound on the number of functions in a normalized Fano CDP that is dependent on the base of the CDP (Theorem \ref{bound}).
This bound is established as follows: after assuming $M=\ZZ^d$,  we pick a point $v_j$ so that $v_j$ and its reflection through the origin $-v_j$ both lie on the $j^{th}$ coordinate axis and in the base polytope. We give both lower and upper bounds on the sum 
\begin{equation}
\label{eq4} 
    \sum_{i=1}^n\left(\Psi_i'(-v_j)+\Psi_i'(v_j)\right), 
\end{equation}
where $\Psi_1,\dots,\Psi_n$ are the functions in a Fano CDP. The lower bound on (\ref{eq4}) follows from Inequality (\ref{eq2}).  
The upper bound on (\ref{eq4}) is provided by Lemma~\ref{boundonsum2}, which uses the concavity of the functions $\Psi_i'$ to provide an upper bound for the sum $\Psi_i'(-v_j)+\Psi_i'(v_j)$.  We sum (\ref{eq4}) over all $j=1,\dots ,d$. By arguing that the upper bound on $\Psi_i'(-v_j)+\Psi_i'(v_j)$ can only be achieved $d-1$ times for fixed $i$, as otherwise $\Psi_i'$ would be linear with integral slope (see Lemmas~\ref{equiv1} and \ref{boundonsum2}), the bound given in the theorem  is obtained.

\subsection{Preliminaries} We establish some straightforward results on Fano CDPs.
We let $\Psi$ be a $(d+1)$-dimensional normalized Fano CDP with base polytope $\Box$ and translated functions $\Psi_1',\dots,\Psi_n'$, which are at height one. For simplicity, we assume that $M=\ZZ^d$.
 First we introduce the notion of \emph{integral} and \emph{non-integral} CDPs:

\begin{defn}
The function $\Psi_i$ is said to be \emph{integral} if $\Psi_i(v)\in\frac{1}{\lambda}\mathbb{Z}$ for all $\lambda\in \NN$, all points $v\in\Box\cap \frac{1}{\lambda} M$. Otherwise we say that $\Psi_i$ is \emph{non-integral.}
\end{defn}

\begin{rem}
If $\Psi_i$ is linear, then it is integral in the above sense if and only if it has integral slope.
\end{rem}

\begin{lemma}\label{lemma:rem}
For any $i=1,\ldots, n$, the function $\Psi_i$ is integral if and only if $\Psi_i'(0)=1$.
\end{lemma}
\begin{proof}
See \cite[Remark 3.7]{stability}
\end{proof}

\begin{lemma}
\label{linearalonglines}
Suppose $\Psi_i'(0)=1$. Then $\Psi_i'$ is linear along the line segment from $0$ to $v$ for all $v\in\partial \Box$. 
\end{lemma}
\begin{proof}
Restrict the base polytope $\Box$ to the segment, which we denote by $L$, and let $F$ be a facet of $\Gamma(\Psi_i')$ containing the point above the origin and some other point of $L$. Suppose $F$ meets another facet, say $F'$, along $L$. Extending $F'$, its value over the origin would be larger than 1 by the concavity of $\Psi_i'$. Then the point $(0,\dots,0,1)$ lies between the extension of $F'$ and the origin, contradicting $F'$ being at height one.
\end{proof}
\begin{lemma}

\label{equiv1}

If $\Psi_i'$ is identically one along every coordinate axis, then $\Psi_i'\equiv1$.  
 
\end{lemma}

\begin{proof}
By concavity of $\Psi_i'$, $\Psi_i'\geq 1$ on the convex hull of the intersection of $\Box$ with the coordinate axes. 
Thus by Lemma~\ref{linearalonglines}, $\Psi_i'\geq1$ along the line from the origin to any point $v\in \Box$, and hence $\Psi_i'\geq1$ on $\Box$.

Suppose that there is some point $v\in \Box$ such that $\Psi'_i(v)>1$. Choosing $\alpha>0$ sufficiently small so that $-\alpha v\in \Box$, the concavity of $\Psi'_i$ would imply that $\Psi'_i(-\alpha v)<1$, a contradiction. Hence $\Psi'_i\equiv1$.
\end{proof}

 

\begin{lemma}

\label{leq1/2}

If $\Psi_i'$ is non-integral, then $\Psi_i'(0)=1/\lambda$ for some $\lambda\in \ZZ_{>1}$. In particular, $\Psi_i'(0)\leq\frac{1}{2}$. 
\end{lemma}

\begin{proof}

The point $p=(0,\dots,0,\Psi_i'(0))$ is in $\Gamma(\Psi_i')$. Since $\Gamma(\Psi_i')$ at height one, there is some $v\in\mathbb{Z}^{d+1}$ such that $\langle p,v\rangle=-1$, that is, 
\[ \Psi_i'(0)=\frac{-1}{\lambda}, \]
where $\lambda=v_{d+1}$ is the $(d+1)^{th}$ component of $v$.
Since $0 < \Psi_i'(0) <1$ by assumption and Lemma \ref{lemma:rem}, and $v_{d+1}\in\ZZ$, we have the desired result.
\end{proof}

\begin{lemma}
\label{boundonsum2}
Let $v\in \RR^d$ lie on one of the coordinate axes, and suppose that $\pm v\in \Box$. 
\begin{enumerate}
\item If $\Psi_i'$ is non-integral, then
\[ \Psi_i'(-v)+\Psi_i'(v)\leq 1. \]
\item If $\Psi_i'$ is integral but non-linear along the line segment from $-v$ to $v$, then 
\[ \Psi_i'(-v)+\Psi_i'(v)\leq 2 -|v|. \]
\item If $\Psi_i'$ is integral and linear along the line segment from $-v$ to $v$, then
\[ \Psi_i'(-v)+\Psi_i'(v)= 2. \]
\end{enumerate}
\end{lemma}

\begin{proof}

First assume that $\Psi_i'$ is non-integral. Then, by the concavity of $\Psi_i'$ and by Lemma~\ref{leq1/2}, we have
$$ \frac{\Psi_i'(-v)+\Psi_i'(v)}{2} \leq \Psi_i'(0) \leq \frac{1}{2} $$
and the result follows.

Suppose $\Psi_i'$ is integral. By shearing we may assume that $\Psi_i'(-v)=1$. If $\Psi_i'$ is linear along the line segment from $-v$ to $v$, then $\Psi_i'(v)=1$, and the result in this case holds. Otherwise suppose $\Psi_i'$ is non-linear along this line segment. Then the slope of $\Psi_i'$ along the line segment from the origin to $v$ is at most $-1$, which gives $\Psi_i'(v)\leq1-|v|$.  Since the bound on the sum $\Psi_i'(-v)+\Psi_i'(v)$ is invariant under shearing, the result follows.
\end{proof}

\subsection{Main Result}

We now show that the number of functions in a normalized Fano CDP is bounded by a value determined only by the base of the CDP. 

Fix a lattice polytope $\Box$ containing the origin in its interior.
For any primitive element $v$ of the lattice $M$, define
\[\alpha_v=\min\{1,\max\{ \alpha\in \RR_{\geq 0}\ |\ \pm\alpha v \in \Box\}\}.\]
\begin{lemma}\label{lemma:bound}
	Let $v\in M$ be primitive.
Any normalized Fano CDP has at most $4/\alpha_v$ functions $\Psi_i$ which are either non-integral, or non-linear along the line spanned by $v$.
\end{lemma}
\begin{proof}
Set $v'=\alpha_vv$.
 Define
\[ R := \{i\ |\ 1\leq i\leq n,\ \textrm{and}\ \Psi_i'(-v')+\Psi_i'(v')<2\} \]
and $r:= \# R$.
Then by Lemmas~\ref{boundonsum2} and \ref{linearalonglines}, $r$ is exactly the number of functions which are either non-integral, or non-linear along the line spanned by $v$.
Note that, for each $i\in R$, we have by Lemma~\ref{boundonsum2} that $\Psi_i'(-v')+\Psi_i'(v') \leq 2-\alpha_v.$ This bound works for non-integral $\Psi_i'$ because $\alpha_v\leq1$. Using this and the lower bound given in Inequality \eqref{eq2}, we have
$$ 2n-4 \leq \sum_{i=1}^n(\Psi_i'(-v')+\Psi_i'(v')) \leq 2(n-r) +r(2-\alpha_v)=2n-\alpha_vr $$
and hence $r\leq \frac{4}{\alpha_v}$.
\end{proof}

Now, fix a basis $e_j$ of the lattice $M$.
We define the constant $c(\Box)$ by
\[
	c(\Box)=\sum_{j\leq d} \frac{4}{\alpha_{e_j}}
.\]

\begin{thm}
\label{bound}
Let $\Psi$ be a $(d+1)$-dimensional normalized Fano CDP with base polytope
$\Box$. Then $\Psi$ has at most ${c(\Box)}$ functions.
\end{thm}

\begin{proof}

Let $\Psi$ and $\Box$ be as above. By Lemma \ref{equiv1}, any integral function that is linear along each coordinate axis is linear; since $\Psi$ is normalized, no such functions exist. Hence, each $\Psi_i$ must be either non-integral, or non-linear along one of the coordinate axes. Applying Lemma \ref{lemma:bound} for $v=e_1,\ldots,e_d$, we obtain 
\[
	n\leq\sum_{j=1}^d  4/\alpha_{e_j}=c(\Box).
\]
\end{proof}

\begin{ex}\label{ex:cross}
	Let $\Box$ be the $d$-dimensional cross-polytope, that is, $\Box$ is the convex hull in $\RR^d$ of $\pm e_1,\ldots,\pm e_d$, where $e_i$ are the standard basis vectors. Then using the standard basis, $c(\Box)=4d$, so by Theorem \ref{bound} any normalized Fano CDP with base $\Box$ has at most $4d$ functions.

	In fact, this bound is sharp. For $j=0,\ldots,d-2$, set
	\[
		\Psi'_{1+4j}(x)=
		\Psi'_{2+4j}(x)=
		\Psi'_{3+4j}(x)=
		\Psi'_{4+4j}(x)=
		\min\{1,x_{j+1}+1\}
	\]
and also set
\begin{align*}
		\Psi'_{4d-3}(x)=
		\Psi'_{4d-2}(x)=
		\Psi'_{4d-1}(x)=
		\min\{1,x_d+1\}\\
		\Psi'_{4d}=\min\{1,x_d+1\}-\sum_{k=1}^d 2x_k.
\end{align*}
These functions are all at height one, and satisfy Inequality \eqref{eq2}, hence they come from a normalized Fano CDP $\Psi$ with $4d$ non-linear functions.
	\end{ex}

\section{Finiteness over Fixed Base}
\label{sec:Base}
In this section, we establish Main Theorem~\ref{thm:2}, which says
that there are only finitely many equivalence classes of Fano CDPs over
a fixed base. We do this by showing in Theorem~\ref{fixedPn} that if
we first fix the number of functions in our CDP, there are only
finitely many possibilities over the given base. Main Theorem~\ref{thm:2} then
follows immediately, given Theorem~\ref{bound}.

The proof of Theorem \ref{fixedPn}, is an argument which reduces
possibilities by considering the \emph{regions of linearity} of a
function.  Let $\Psi$ be a CDP with base $\Box$ and functions
$\Psi_1,\dots,\Psi_n$. Each of the functions $\Psi_i$ are
piecewise linear and concave. Consequently, the facets of
$\Gamma(\Psi_i)$ are convex, and so project down to a subdivision of
$\Box$ into polytopes. Moreover, as $\Gamma(\Psi_i)$ has integral
vertices, the vertices of the polytopes in the decomposition are
integral. Hence $\Gamma(\Psi_i)$ induces a subdivision of the base
polytope into a union of finitely many lattice polytopes. We say that
a $d$-dimensional polytope in this subdivision is a \emph{region of
  linearity} of the function $\Psi_i$, where $d$ is the dimension of
the base polytope~$\Box$.

\begin{thm}
\label{fixedPn}
Let $\Box$ be a fixed lattice polytope and $n$ a fixed positive integer. Up to equivalence, there are only finitely many Fano CDPs with base polytope $\Box$ and $n$ functions. 
\end{thm}

\begin{proof}
  Suppose $\Psi$ is a Fano CDP with base polytope $\Box$ and
  translated functions $\Psi_1',\dots,\Psi_n'$. By shearing and by
  using Inequality~\eqref{eq2} we can bound $\Psi_n'(u)$ at each point
  $u\in\Box$ as follows: 

\begin{step}
Fix regions of linearity. 
\end{step}
\noindent
As discussed above, each $\Psi_i'$ induces a subdivision of $\Box$ into its regions of linearity; the set of these regions we denote by $R_i$. Since $\Box$ contains only finitely many lattice points, there are only finitely many possibilities for the sets $R_i$, so we may assume that we have fixed them once and for all.

\begin{step} 
Give upper bounds for $\Psi_1',\dots,\Psi_{n-1}'$. 
 \end{step}
\noindent
Consider the polytopes given by intersections of the form $P_{1}\cap\cdots\cap P_{n-1}$, where each $P_{i}\in R_i$. Let $P$ be one of the $d$-dimensional polytopes containing the origin obtained through this process. Note that the restriction of each $\Psi_i'$ to $P$ is a linear function for $i=1,\dots, n-1$. Let $C$ be the cone generated by the elements of $P$. By \cite[Theorem 11.1.9]{CLS}, it contains a lattice basis $v_1,\dots,v_d\in M$. Any element of $C$ is a scalar multiple of a point in $P$, and hence $P$ contains points $p_1,\dots,p_d$ where $p_i=\alpha_iv_i$ for some $\alpha_i\in\mathbb{R}$. As $P$ is convex, we can assume that $|\alpha_i|\leq1$.

The unimodular basis $v_1,\dots,v_d$ corresponds to a basis $v_1^*,\dots,v_d^*$ in $M^*$ so that $v_i^*(v_j)=\delta_{ij}$.  Thus, if we shear $\Psi_i'$ using $v_j^*$, then $\Psi_i'(p_k)$ stays fixed for $k\neq j$. Hence we may assume that $0\leq\Psi_i'(p_j)< 1$ for all $i=1,\dots, n-1$ and $j=1,\dots ,d$. 

As the restriction of $\Psi_i'$ to $P$ is a linear function, it is determined by its values at the points in the set $\{0,p_1,\dots,p_d\}$. Let $L_i$ be the extension of this function to all of $\Box$. By the concavity of $\Psi_i'$, the maximum of $L_i$ bounds $\Psi_i'$ by above.  Consider the set of linear functions $\phi$ defined on $\Box$ such that for each point $p\in\{0,p_1,\dots,p_d\}$ either $\phi(p)=0$ or $\phi(p)=1$. There are only finitely many such functions, and the maximum value obtained by them provides an upper bound for $L_i$ and hence for $\Psi_i'$.

\begin{step}  
Give lower bound for $\Psi_n'$.
\end{step}
\noindent
The upper bounds for $\Psi_1',\dots,\Psi_{n-1}'$ give a lower bound for $\Psi_n'$, through use of Inequality~(\ref{eq2}), that is, $n-2\leq\sum \Psi_i'$.

\begin{step}
Give upper bound for $\Psi_n'$. 
\end{step} 
\noindent
Let $u\in\Box$. Let $\alpha>0$ be sufficiently small so that $v=-\alpha u$ and the origin are in the same region of linearity of $\Psi_n'$. By the concavity of $\Psi_n'$, the line through the points $(0,\Psi_n'(0))$ and $(v,\Psi_n'(v))$ provides an upper bound for $\Psi_n'(u)$. Moreover, this upper bound is maximized by increasing the value for $\Psi_n'(0)$ and decreasing the value for $\Psi_n'(v)$. Thus the upper bound $\Psi_n'(0)\leq1$ and the lower bound for $\Psi_n'(v)$ given in Step 3 provides an upper bound for $\Psi_n'(u)$. 

\begin{step}
Give lower bounds for $\Psi_1',\ldots,\Psi_{n-1}'$.
\end{step}
\noindent
The upper bounds for $\Psi_1',\dots,\Psi_n'$ yield lower bounds for $\Psi_1',\dots,\Psi_{n-1}'$, again by Inequality~(\ref{eq2}).

Having bounded the functions $\Psi_i'$ both from above and below, we may now conclude the proof of the theorem:
\begin{step}
Conclude finiteness by counting possible vertices of graphs.
\end{step}
\noindent
As the vertices of the graph of each $\Psi_i'$ are integral by definition, and the function values here are bounded above and below, there are only finitely many ways to choose the vertices of $\Gamma(\Psi_i')$. Since $\Gamma(\Psi_i')$ is determined by its vertices, there are only finitely many possibilities for the $\Psi_i'$. 

\end{proof}

Combining this result with Theorem \ref{bound} gives the statement of Main Theorem~\ref{thm:2}:

\begin{thm}
\label{fixedBase}
Over any fixed base $\Box$, there are at most finitely many equivalence classes of Fano CDPs.
\end{thm}

\section{Two-Dimensional Fano $T$-Varieties}\label{sec:surfaces}
In this section, we enumerate the non-toric two-dimensional Fano
CDP equivalence classes:
\begin{thm}\label{thm:2d}
There are exactly 34 equivalence classes of non-toric two-dimensional Fano CDPs, as listed in Tables ~\ref{table:en1}~through~\ref{table:en6}.
\end{thm}
\noindent This result was previously obtained by Huggenberger with other methods \cite{huggenberger,hausen:13a}. 

Our enumeration strategy will include choosing  preferred representatives for each equivalence class.
Throughout this section we assume that CDPs are normalized, as
per Definition~\ref{def:normalized}.
The base $\Box$ must include both $1$ and $-1$, so by Theorem \ref{bound}, our normalized CDP has at most four functions.


\begin{lemma}
\label{onenonneg}
Suppose $\Psi_i'$ is a translated function in a normalized two-dimensional Fano
CDP. Then either
$\Psi_i'(-1)\leq0$ or $\Psi_i'(1)\leq0$.
\end{lemma}

\begin{proof}
  Assume, towards contradiction, that both $\Psi_i'(-1)>0$ and
  $\Psi_i'(1)>0$.

  If $\Psi_i'(0)$ is non-integral, then 0 is not a vertex of
  $\Gamma(\Psi_i')$ and thus $(-1,\Psi_i'(-1))$ and $(1,\Psi_i'(1))$ lie on
  the same facet $\Gamma(\Psi_i')$. By the height one restriction there are
  $u,v\in\mathbb{Z}$ such that $\langle u,v\rangle\cdot\langle 1$,
  $\Psi_i'(-1)\rangle=1$ and $\langle u,v\rangle\cdot\langle
  1$,$\Psi_i'(1)\rangle=-1$. Under the constraints of positivity and
  $u,v\in\mathbb{Z}$, we deduce $ \Psi_i'(-1)= \Psi_i'(1)$, and the facet is a horizontal
  line at $\Psi_i'(0)$. By assumption, $\Psi_i'(0)$ is non-integral,
  and so this facet has a non-integral vertex, which is impossible. 

  If, on the other hand, the value $\Psi_i'(0)$ is integral, then by
  Lemma~\ref{lemma:rem}, $\Psi_i'(0)= 1$. Convexity, and positivity
  force $\Psi_i'(-1)=\Psi_i'(1)=1$. From Lemma~\ref{equiv1} it follows that
  $\Psi_i'\equiv1$, which directly contradicts the hypothesis that 
  $\Psi_i'$ is normalized.
  \end{proof}

\begin{lemma}
\label{lemma:it}
Suppose $\Psi_i'$ is a translated function in a normalized two-dimensional Fano CDP. If $\Psi_i'(1)>-1$, then
$\Psi_i'(-1)\leq 1$. Similarly if $\Psi_i'(-1)>-1$, then
$\Psi_i'(1)\leq1$.
\end{lemma}

\begin{proof}
  If $\Psi_i'(1)>-1$ and $\Psi_i'(1)>1$, or vice versa, then $\Psi_i'$
  may be sheared so that both values are positive, contradicting
  Lemma~\ref{onenonneg}.
\end{proof}

We can use shearing within an equivalence class to determine a preferred value of
the first $n-1$ functions at $-1$, and then use these lemmas to deduce
additional restrictions. First, we shear the functions so that
$\Psi_i'(-1)$ is non-negative and in a preferred range:
\begin{equation}\tag{E1}
	0\leq\Psi_i'(-1)<1,\quad i=1\dots n-1.\label{E1}
\end{equation}
%
Using  Lemma~\ref{lemma:it} we deduce that 
\begin{equation}\tag{E2}\label{E2}
\Psi_i'(1)\leq1, \quad i=1\dots n-1. 
\end{equation}
We can argue about the final function by using the condition of
Inequality~(\ref{eq1}), namely $\sum_{i=1}^n \Psi_i'
> n-2$ on $\Box^\circ$. Consequently, 
\begin{equation}\tag{E3}\label{E3}
-1<\Psi_n'(-1)\quad\text{and}\quad -1\leq \Psi_n'(1).
\end{equation}
The latter inequality is strict if $1$ is in the interior of $\Box$. 
Finally, from this, and Lemma~\ref{lemma:it} we conclude that if $1\in\Box^\circ$,
\begin{equation}\tag{E4}\label{E4}
\Psi_n'(-1)\leq1\quad\text{and}\quad\Psi_n'(1)\leq1.
\end{equation}
These bounds severely restrict the possible base polytope for our representatives.
\begin{prop}
\label{prop:noint}
Under the above restrictions on the form of the functions, there are no normalized two-dimensional non-toric Fano CDPs with a base $\Box$ such that $-1,1\in\Box^\circ$.
\end{prop}

\begin{proof}
  Suppose $\Psi_1',\dots,\Psi_n'$ are the translated functions of a Fano CDP
  $\Psi$ with such a base $\Box$. Assume \eqref{E1}, \eqref{E2}, \eqref{E3}, and \eqref{E4}.  We can
  also assume $\Psi_n'(u)\leq0$ for either $u=-1$ or $u=1$, given Lemma~\ref{onenonneg}. By \eqref{E1} or \eqref{E2}, if there
  is some $j<n$ such that $\Psi_j'(u)\leq0$, then
  inequality~(\ref{eq1}), $n-2<\sum_{i=1}^{n-1}\Psi_i'(u)$, cannot be
  satisfied. Hence $\Psi_i'(u)>0$ for all
  $i=1\dots n-1$, and so by Lemma~\ref{onenonneg}, $\Psi_i'(-u)\leq0$
  for all $i=1\dots n-1$. Combining inequalities~(\ref{eq1}) and \eqref{E4} gives
  $n-2<\Psi_n'(-u)\leq 1$ from which we deduce that~$n<3$. Thus $\Psi$
  is toric.
\end{proof}

We now fix a normalized two-dimensional non-toric Fano CDP $\Psi=(\Psi_1,\ldots,\Psi_n)$ with base $\Box$. 
The base polytope $\Box$ thus has at least one of $-1$ or $1$ as boundary
points. Without loss of generality, in our choice representative we assume that $-1$ is a boundary
point of $\Box$. In the following, we further assume that
$\Psi_1'(-1)=\cdots=\Psi_{n-1}'(-1)=0$ (again, shearing if necessary).

The next parameter that we consider is the length of the base
polytope, which we denote by $m$. If $m=2$, then $\Box=[-1,1]$. As mentioned above, Theorem~\ref{bound} gives the bound
$n\leq 4$. We have listed all two-dimensional Fano CDPs over $\Box=[-1,1]$ with three and
four functions in Tables~\ref{table:en1} and~\ref{table:2},
respectively. If $m>2$, then $1$ is an interior point of $\Box$. The
inequality computed in Theorem~\ref{bound} is strict, which yields a bound $n<4$. In this case ($m>2$), our normalized CDP has exactly three functions
$\Psi_1,\Psi_2,\Psi_3$.

\begin{lemma}
  Suppose that $\Box$ has length larger than two.  If $\Psi_i'(0)$ is
  integral, then $\Psi_i'$ has at most one interior vertex.
\end{lemma}
\begin{proof}
  It follows directly from Lemma~\ref{linearalonglines} that the only
  vertex must live above~$0$.
\end{proof}

\begin{lemma}
  Suppose that $\Box = [-1, m-1]$ . For $i=1, 2$,
  if $\Psi_i'(0)$ is non-integral, then up to shearing, either
\begin{enumerate}
\item $\Psi_i'$ 
  has a single vertex at some integer $\lambda-1$ satisfying $1 < \lambda<m$:
\[
\Psi_i'(x)=\begin{cases}
(x+1)/\lambda &x+1\leq \lambda\\
1 & \lambda \leq x+1 \leq m;
\end{cases}
\]
\item or $\Psi_i'$ is simply a line with slope $\lambda \mid m$:
\[   \Psi_i'(x)=(x+1)/\lambda.\]
\end{enumerate}
\end{lemma}
\begin{proof}
  We have determined that $\Psi_i'(-1)=0$ and by Lemma~\ref{leq1/2}
  there is a positive integer $\lambda$ so that
  $\Psi_i'(0)=\frac{1}{\lambda}$. Either $\Psi_i'$ is a line, or the
  point $(\lambda-1, 1)$ is a vertex. If it is not a line, then
  $\Psi_i'\equiv 1$ on $[\lambda-1, m-1]$ to ensure that the facet is at height one.
 \end{proof}

The shape of $\Psi_3'$ can similarly be described. The normalization of
the other functions gives that  $\Psi_3'(-1)=1$.  Furthermore, recall that $\Psi_3'(1)>-1$ (by \eqref{E3}). 

Next we consider the possible values at the boundary of the polytope, $m-1$. The
major restricting factor is that the regions of linearity start and end at lattice
points, and the functions remain at height one. We next show that the nature of the
equations limits $m$ to be 3, 4 or 6, and with this information we can construct
all cases by iterating through possible choices of shape, and
$\lambda$ for the three functions.

To bound $m$ we look closer at the inequalities. 
These lemmas determine the possible values for $\Psi_1'(m-1)$,
$\Psi_2'(m-1)$, and $\Psi_3'(m-1)$. For $i=1,2$, 
either $\Psi_i'(m-1)= \frac{m}{\lambda}$ for some  $1<\lambda$ which
divides $m$, or $\Psi_i'(m-1)=\{-k(m-1)+1\}$ for some
non-negative integer~$k$. 

The permissible values for $\Psi_3'(m-1)$ can be found by subtracting
$m-1$ from the permissible values of $\Psi_1'(m-1)$. Moreover, we only
keep the values that imply $\Psi_3'(1)>-1$ and deduce
\[\Psi_3'(m-1)\in\left\{-m+2,-\frac{\mu}{\mu+1}m+1\right\} \]
for some  $0\leq\mu$ and $\mu+1\mid m$.
We substitute these possibilities into Equation~(\ref{eq3}), and
search for integer solutions to
\[ \Psi_1'(m-1)+\Psi_2'(m-1)+\Psi'_3(m-1)=1.\]

We find a finite number of integral solutions to this equation, and
find that in these solutions, $m\in\{3,4,6\}$. We illustrate one of
these computations in the next example.  The equivalence classes of
Fano CDPs with three functions and base of length $3$, $4$, and $6$
are given in Tables~\ref{table:en3}, \ref{table:en4}, and
\ref{table:en6} respectively.
This completes the proof of Theorem \ref{thm:2d} (and Main Theorem \ref{thm:3}).

\begin{example}
Suppose that $\Psi_1'(m-1)=\frac{m}{\lambda_1}$,
  $\Psi_2'(m-1)=\frac{m}{\lambda_2}$, and $\Psi_3'(m-1)=-m+2$. We seek
  integer solutions to the equation
\[ \frac{m}{\lambda_1}+\frac{m}{\lambda_2}-m+2=1. \]

Simplifying to $\frac{m}{\lambda_1}+\frac{m}{\lambda_2}=m-1$, it is easy to see that either $\lambda_1=2$ or $\lambda_2=2$. 
Without loss of generality, assume $\lambda_1=2$.  It follows that $m\mid 2\lambda_2$. Thus $\frac{m}{\lambda_2}=1$ or $\frac{m}{\lambda_2}=2$. With the values for $\frac{m}{\lambda_2}$ and $\lambda_1$ fixed, we are able to solve for $m$. The resulting solutions are $(m,\lambda_1,\lambda_2)\in\{(6,2,3),(4,2,4)\}$.

Enumerating all Fano CDPs satisfying $(m,\lambda_1,\lambda_2)=(6,2,3)$ leads to
all equivalence classes of Fano CDPs given in
Table~\ref{table:en6}. Enumerating all Fano CDPs satisfying
$(m,\lambda_1,\lambda_2)=(4,2,4)$ leads to all equivalence classes of Fano CDPs
given in Table~\ref{table:en4}.
\end{example}

\begin{table}
\centering
\begin{tabular}{c|c|c|c}
Number & $\Psi_1'$ & $\Psi_2'$ & $\Psi_3'$ \\ 
\hline
1 &
\begin{tikzpicture}[baseline=-.65ex,scale=0.5]
\draw[step=1cm,gray,very thin] (-1,-1) grid (1,1);
\draw[gray, line width=1] (-1,0) -- (1,0);
\filldraw[gray] (-1,0) circle (2pt);
\filldraw[gray] (0,0) circle (2pt);
\filldraw[gray] (1,0) circle (2pt);
\draw[black, line width=1] (-1,0)  -- (1,1);
\end{tikzpicture} & \begin{tikzpicture}[baseline=-.65ex,scale=0.5]
\draw[step=1cm,gray,very thin] (-1,-1) grid (1,1);
\draw[gray, line width=1] (-1,0) -- (1,0);
\filldraw[gray] (-1,0) circle (2pt);
\filldraw[gray] (0,0) circle (2pt);
\filldraw[gray] (1,0) circle (2pt);
\draw[black, line width=1] (-1,0) -- (1,1);
\end{tikzpicture} &
\begin{tikzpicture}[baseline=-.65ex,scale=0.5]
\draw[step=1cm,gray,very thin] (-1,-1) grid (1,1);
\draw[gray, line width=1] (-1,0) -- (1,0);
\filldraw[gray] (-1,0) circle (2pt);
\filldraw[gray] (0,0) circle (2pt);
\filldraw[gray] (1,0) circle (2pt);
\draw[black, line width=1] (-1,1) --  (1,0);
\end{tikzpicture} \\
\hline
2 &
\begin{tikzpicture}[baseline=-.65ex,scale=0.5]
\draw[step=1cm,gray,very thin] (-1,-1) grid (1,1);
\draw[gray, line width=1] (-1,0) -- (1,0);
\filldraw[gray] (-1,0) circle (2pt);
\filldraw[gray] (0,0) circle (2pt);
\filldraw[gray] (1,0) circle (2pt);
\draw[black, line width=1] (-1,0)  -- (1,1);
\end{tikzpicture} & \begin{tikzpicture}[baseline=-.65ex,scale=0.5]
\draw[step=1cm,gray,very thin] (-1,-1) grid (1,1);
\draw[gray, line width=1] (-1,0) -- (1,0);
\filldraw[gray] (-1,0) circle (2pt);
\filldraw[gray] (0,0) circle (2pt);
\filldraw[gray] (1,0) circle (2pt);
\draw[black, line width=1] (-1,0) -- (1,1);
\end{tikzpicture} &
\begin{tikzpicture}[baseline=-.65ex,scale=0.5]
\draw[step=1cm,gray,very thin] (-1,-1) grid (1,1);
\draw[gray, line width=1] (-1,0) -- (1,0);
\filldraw[gray] (-1,0) circle (2pt);
\filldraw[gray] (0,0) circle (2pt);
\filldraw[gray] (1,0) circle (2pt);
\draw[black, line width=1] (-1,1) -- (0,1) -- (1,0);
\end{tikzpicture} \\
\hline
3 &
\begin{tikzpicture}[baseline=-.65ex,scale=0.5]
\draw[step=1cm,gray,very thin] (-1,-1) grid (1,1);
\draw[gray, line width=1] (-1,0) -- (1,0);
\filldraw[gray] (-1,0) circle (2pt);
\filldraw[gray] (0,0) circle (2pt);
\filldraw[gray] (1,0) circle (2pt);
\draw[black, line width=1] (-1,0)  -- (1,1);
\end{tikzpicture} & \begin{tikzpicture}[baseline=-.65ex,scale=0.5]
\draw[step=1cm,gray,very thin] (-1,-1) grid (1,1);
\draw[gray, line width=1] (-1,0) -- (1,0);
\filldraw[gray] (-1,0) circle (2pt);
\filldraw[gray] (0,0) circle (2pt);
\filldraw[gray] (1,0) circle (2pt);
\draw[black, line width=1] (-1,0) -- (0,1) -- (1,0);
\end{tikzpicture} &
\begin{tikzpicture}[baseline=-.65ex,scale=0.5]
\draw[step=1cm,gray,very thin] (-1,-1) grid (1,1);
\draw[gray, line width=1] (-1,0) -- (1,0);
\filldraw[gray] (-1,0) circle (2pt);
\filldraw[gray] (0,0) circle (2pt);
\filldraw[gray] (1,0) circle (2pt);
\draw[black, line width=1] (-1,1) --  (1,0);
\end{tikzpicture} \\
\hline
4 &
\begin{tikzpicture}[baseline=-.65ex,scale=0.5]
\draw[step=1cm,gray,very thin] (-1,-1) grid (1,1);
\draw[gray, line width=1] (-1,0) -- (1,0);
\filldraw[gray] (-1,0) circle (2pt);
\filldraw[gray] (0,0) circle (2pt);
\filldraw[gray] (1,0) circle (2pt);
\draw[black, line width=1] (-1,0)  -- (0,1) -- (1,1);
\end{tikzpicture} & \begin{tikzpicture}[baseline=-.65ex,scale=0.5]
\draw[step=1cm,gray,very thin] (-1,-1) grid (1,1);
\draw[gray, line width=1] (-1,0) -- (1,0);
\filldraw[gray] (-1,0) circle (2pt);
\filldraw[gray] (0,0) circle (2pt);
\filldraw[gray] (1,0) circle (2pt);
\draw[black, line width=1] (-1,0) -- (0,1) -- (1,0);
\end{tikzpicture} &
\begin{tikzpicture}[baseline=-.65ex,scale=0.5]
\draw[step=1cm,gray,very thin] (-1,-1) grid (1,1);
\draw[gray, line width=1] (-1,0) -- (1,0);
\filldraw[gray] (-1,0) circle (2pt);
\filldraw[gray] (0,0) circle (2pt);
\filldraw[gray] (1,0) circle (2pt);
\draw[black, line width=1] (-1,1) -- (0,1) --  (1,0);
\end{tikzpicture} \\
\hline
5 &
\begin{tikzpicture}[baseline=-.65ex,scale=0.5]
\draw[step=1cm,gray,very thin] (-1,-1) grid (1,1);
\draw[gray, line width=1] (-1,0) -- (1,0);
\filldraw[gray] (-1,0) circle (2pt);
\filldraw[gray] (0,0) circle (2pt);
\filldraw[gray] (1,0) circle (2pt);
\draw[black, line width=1] (-1,0)  -- (0,1) -- (1,1);
\end{tikzpicture} & \begin{tikzpicture}[baseline=-.65ex,scale=0.5]
\draw[step=1cm,gray,very thin] (-1,-1) grid (1,1);
\draw[gray, line width=1] (-1,0) -- (1,0);
\filldraw[gray] (-1,0) circle (2pt);
\filldraw[gray] (0,0) circle (2pt);
\filldraw[gray] (1,0) circle (2pt);
\draw[black, line width=1] (-1,0) -- (0,1) -- (1,0);
\end{tikzpicture} &
\begin{tikzpicture}[baseline=-.65ex,scale=0.5]
\draw[step=1cm,gray,very thin] (-1,-1) grid (1,1);
\draw[gray, line width=1] (-1,0) -- (1,0);
\filldraw[gray] (-1,0) circle (2pt);
\filldraw[gray] (0,0) circle (2pt);
\filldraw[gray] (1,0) circle (2pt);
\draw[black, line width=1] (-1,1) --  (1,0);
\end{tikzpicture} \\
\hline
6 &
\begin{tikzpicture}[baseline=-.65ex,scale=0.5]
\draw[step=1cm,gray,very thin] (-1,-1) grid (1,1);
\draw[gray, line width=1] (-1,0) -- (1,0);
\filldraw[gray] (-1,0) circle (2pt);
\filldraw[gray] (0,0) circle (2pt);
\filldraw[gray] (1,0) circle (2pt);
\draw[black, line width=1] (-1,0) -- (0,1) -- (1,1);
\end{tikzpicture} & \begin{tikzpicture}[baseline=-.65ex,scale=0.5]
\draw[step=1cm,gray,very thin] (-1,-1) grid (1,1);
\draw[gray, line width=1] (-1,0) -- (1,0);
\filldraw[gray] (-1,0) circle (2pt);
\filldraw[gray] (0,0) circle (2pt);
\filldraw[gray] (1,0) circle (2pt);
\draw[black, line width=1] (-1,0) -- (0,1) -- (1,1);
\end{tikzpicture} & \begin{tikzpicture}[baseline=-.65ex,scale=0.5]
\draw[step=1cm,gray,very thin] (-1,-1) grid (1,1);
\draw[gray, line width=1] (-1,0) -- (1,0);
\filldraw[gray] (-1,0) circle (2pt);
\filldraw[gray] (0,0) circle (2pt);
\filldraw[gray] (1,0) circle (2pt);
\draw[black, line width=1] (-1,1) -- (0,1) -- (1,0);
\end{tikzpicture} \\ 
\hline
7 &
\begin{tikzpicture}[baseline=-.65ex,scale=0.5]
\draw[step=1cm,gray,very thin] (-1,-1) grid (1,1);
\draw[gray, line width=1] (-1,0) -- (1,0);
\filldraw[gray] (-1,0) circle (2pt);
\filldraw[gray] (0,0) circle (2pt);
\filldraw[gray] (1,0) circle (2pt);
\draw[black, line width=1] (-1,0) -- (0,1) -- (1,1);
\end{tikzpicture} & \begin{tikzpicture}[baseline=-.65ex,scale=0.5]
\draw[step=1cm,gray,very thin] (-1,-1) grid (1,1);
\draw[gray, line width=1] (-1,0) -- (1,0);
\filldraw[gray] (-1,0) circle (2pt);
\filldraw[gray] (0,0) circle (2pt);
\filldraw[gray] (1,0) circle (2pt);
\draw[black, line width=1] (-1,0) -- (0,1) -- (1,1);
\end{tikzpicture} &
\begin{tikzpicture}[baseline=-.65ex,scale=0.5]
\draw[step=1cm,gray,very thin] (-1,-1) grid (1,1);
\draw[gray, line width=1] (-1,0) -- (1,0);
\filldraw[gray] (-1,0) circle (2pt);
\filldraw[gray] (0,0) circle (2pt);
\filldraw[gray] (1,0) circle (2pt);
\draw[black, line width=1] (-1,1) -- (1,0);
\end{tikzpicture} 
\end{tabular}
\smallskip

 \caption{Base polytope length $m=2$, number of functions $n=3$}
 \label{table:en1}
\end{table}

\begin{table}
\centering
\begin{tabular}{c|c|c|c|c}
Number & $\Psi_1'$ & $\Psi_2'$ & $\Psi_3'$ & $\Psi_4'$ \\ 
 \hline
 8 &
\begin{tikzpicture}[baseline=-.65ex,scale=0.5]
\draw[step=1cm,gray,very thin] (-1,-1) grid (1,1);
\draw[gray, line width=1] (-1,0) -- (1,0);
\filldraw[gray] (-1,0) circle (2pt);
\filldraw[gray] (0,0) circle (2pt);
\filldraw[gray] (1,0) circle (2pt);
\draw[black, line width=1] (-1,0)  -- (1,1);
\end{tikzpicture} & \begin{tikzpicture}[baseline=-.65ex,scale=0.5]
\draw[step=1cm,gray,very thin] (-1,-1) grid (1,1);
\draw[gray, line width=1] (-1,0) -- (1,0);
\filldraw[gray] (-1,0) circle (2pt);
\filldraw[gray] (0,0) circle (2pt);
\filldraw[gray] (1,0) circle (2pt);
\draw[black, line width=1] (-1,0) -- (1,1);
\end{tikzpicture} & \begin{tikzpicture}[baseline=-.65ex,scale=0.5]
\draw[step=1cm,gray,very thin] (-1,-1) grid (1,1);
\draw[gray, line width=1] (-1,0) -- (1,0);
\filldraw[gray] (-1,0) circle (2pt);
\filldraw[gray] (0,0) circle (2pt);
\filldraw[gray] (1,0) circle (2pt);
\draw[black, line width=1] (-1,0)  -- (1,1);
\end{tikzpicture} &
\begin{tikzpicture}[baseline=-.65ex,scale=0.5]
\draw[step=1cm,gray,very thin] (-1,-1) grid (1,2);
\draw[gray, line width=1] (-1,0) -- (1,0);
\filldraw[gray] (-1,0) circle (2pt);
\filldraw[gray] (0,0) circle (2pt);
\filldraw[gray] (1,0) circle (2pt);
\draw[black, line width=1] (-1,2)  -- (0,1) -- (1,-1);
\end{tikzpicture}\\ 
\hline
 9 &
\begin{tikzpicture}[baseline=-.65ex,scale=0.5]
\draw[step=1cm,gray,very thin] (-1,-1) grid (1,1);
\draw[gray, line width=1] (-1,0) -- (1,0);
\filldraw[gray] (-1,0) circle (2pt);
\filldraw[gray] (0,0) circle (2pt);
\filldraw[gray] (1,0) circle (2pt);
\draw[black, line width=1] (-1,0)  -- (1,1);
\end{tikzpicture} & \begin{tikzpicture}[baseline=-.65ex,scale=0.5]
\draw[step=1cm,gray,very thin] (-1,-1) grid (1,1);
\draw[gray, line width=1] (-1,0) -- (1,0);
\filldraw[gray] (-1,0) circle (2pt);
\filldraw[gray] (0,0) circle (2pt);
\filldraw[gray] (1,0) circle (2pt);
\draw[black, line width=1] (-1,0) -- (1,1);
\end{tikzpicture} & \begin{tikzpicture}[baseline=-.65ex,scale=0.5]
\draw[step=1cm,gray,very thin] (-1,-1) grid (1,1);
\draw[gray, line width=1] (-1,0) -- (1,0);
\filldraw[gray] (-1,0) circle (2pt);
\filldraw[gray] (0,0) circle (2pt);
\filldraw[gray] (1,0) circle (2pt);
\draw[black, line width=1] (-1,0)  -- (0,1) -- (1,1);
\end{tikzpicture} &
\begin{tikzpicture}[baseline=-.65ex,scale=0.5]
\draw[step=1cm,gray,very thin] (-1,-1) grid (1,2);
\draw[gray, line width=1] (-1,0) -- (1,0);
\filldraw[gray] (-1,0) circle (2pt);
\filldraw[gray] (0,0) circle (2pt);
\filldraw[gray] (1,0) circle (2pt);
\draw[black, line width=1] (-1,2) -- (0,1) -- (1,-1);
\end{tikzpicture}\\ 
\hline
 10 &
\begin{tikzpicture}[baseline=-.65ex,scale=0.5]
\draw[step=1cm,gray,very thin] (-1,-1) grid (1,1);
\draw[gray, line width=1] (-1,0) -- (1,0);
\filldraw[gray] (-1,0) circle (2pt);
\filldraw[gray] (0,0) circle (2pt);
\filldraw[gray] (1,0) circle (2pt);
\draw[black, line width=1] (-1,0)  -- (1,1);
\end{tikzpicture} & \begin{tikzpicture}[baseline=-.65ex,scale=0.5]
\draw[step=1cm,gray,very thin] (-1,-1) grid (1,1);
\draw[gray, line width=1] (-1,0) -- (1,0);
\filldraw[gray] (-1,0) circle (2pt);
\filldraw[gray] (0,0) circle (2pt);
\filldraw[gray] (1,0) circle (2pt);
\draw[black, line width=1] (-1,0) -- (0,1) -- (1,1);
\end{tikzpicture} & \begin{tikzpicture}[baseline=-.65ex,scale=0.5]
\draw[step=1cm,gray,very thin] (-1,-1) grid (1,1);
\draw[gray, line width=1] (-1,0) -- (1,0);
\filldraw[gray] (-1,0) circle (2pt);
\filldraw[gray] (0,0) circle (2pt);
\filldraw[gray] (1,0) circle (2pt);
\draw[black, line width=1] (-1,0)  -- (0,1) -- (1,1);
\end{tikzpicture} &
\begin{tikzpicture}[baseline=-.65ex,scale=0.5]
\draw[step=1cm,gray,very thin] (-1,-1) grid (1,2);
\draw[gray, line width=1] (-1,0) -- (1,0);
\filldraw[gray] (-1,0) circle (2pt);
\filldraw[gray] (0,0) circle (2pt);
\filldraw[gray] (1,0) circle (2pt);
\draw[black, line width=1] (-1,2)  -- (0,1) -- (1,-1);
\end{tikzpicture}\\ 
\hline
 11 &
\begin{tikzpicture}[baseline=-.65ex,scale=0.5]
\draw[step=1cm,gray,very thin] (-1,-1) grid (1,1);
\draw[gray, line width=1] (-1,0) -- (1,0);
\filldraw[gray] (-1,0) circle (2pt);
\filldraw[gray] (0,0) circle (2pt);
\filldraw[gray] (1,0) circle (2pt);
\draw[black, line width=1] (-1,0)  -- (0,1) -- (1,1);
\end{tikzpicture} & \begin{tikzpicture}[baseline=-.65ex,scale=0.5]
\draw[step=1cm,gray,very thin] (-1,-1) grid (1,1);
\draw[gray, line width=1] (-1,0) -- (1,0);
\filldraw[gray] (-1,0) circle (2pt);
\filldraw[gray] (0,0) circle (2pt);
\filldraw[gray] (1,0) circle (2pt);
\draw[black, line width=1] (-1,0) -- (0,1) -- (1,1);
\end{tikzpicture} & \begin{tikzpicture}[baseline=-.65ex,scale=0.5]
\draw[step=1cm,gray,very thin] (-1,-1) grid (1,1);
\draw[gray, line width=1] (-1,0) -- (1,0);
\filldraw[gray] (-1,0) circle (2pt);
\filldraw[gray] (0,0) circle (2pt);
\filldraw[gray] (1,0) circle (2pt);
\draw[black, line width=1] (-1,0)  -- (0,1) -- (1,1);
\end{tikzpicture} &
\begin{tikzpicture}[baseline=-.65ex,scale=0.5]
\draw[step=1cm,gray,very thin] (-1,-1) grid (1,2);
\draw[gray, line width=1] (-1,0) -- (1,0);
\filldraw[gray] (-1,0) circle (2pt);
\filldraw[gray] (0,0) circle (2pt);
\filldraw[gray] (1,0) circle (2pt);
\draw[black, line width=1] (-1,2)  -- (0,1) -- (1,-1);
\end{tikzpicture}\\ 
\hline
\end{tabular}
 \caption[]{Base polytope length $m=2$, number of functions $n=4$}
 \label{table:2}
\end{table}

\begin{table}
\centering
\begin{tabular}{c|c|c|c}
Number & $\Psi_1'$ & $\Psi_2'$ & $\Psi_3'$\\ 
 \hline
 12 &
 \begin{tikzpicture}[baseline=-.65ex,scale=0.5]
\draw[step=1cm,gray,very thin] (-1,-1) grid (2,1);
\draw[gray, line width=1] (-1,0) -- (2,0);
\filldraw[gray] (-1,0) circle (2pt);
\filldraw[gray] (0,0) circle (2pt);
\filldraw[gray] (1,0) circle (2pt);
\filldraw[gray] (2,0) circle (2pt);
\draw[black, line width=1] (-1,0) -- (0,1) -- (2,1);
\end{tikzpicture} &
 \begin{tikzpicture}[baseline=-.65ex,scale=0.5]
\draw[step=1cm,gray,very thin] (-1,-1) grid (2,1);
\draw[gray, line width=1] (-1,0) -- (2,0);
\filldraw[gray] (-1,0) circle (2pt);
\filldraw[gray] (0,0) circle (2pt);
\filldraw[gray] (1,0) circle (2pt);
\filldraw[gray] (2,0) circle (2pt);
\draw[black, line width=1] (-1,0) -- (0,1) -- (2,1);
\end{tikzpicture} & 
 \begin{tikzpicture}[baseline=-.65ex,scale=0.5]
\draw[step=1cm,gray,very thin] (-1,-1) grid (2,1);
\draw[gray, line width=1] (-1,0) -- (2,0);
\filldraw[gray] (-1,0) circle (2pt);
\filldraw[gray] (0,0) circle (2pt);
\filldraw[gray] (1,0) circle (2pt);
\filldraw[gray] (2,0) circle (2pt);
\draw[black, line width=1] (-1,1) -- (2,-1);
\end{tikzpicture} \\
 \hline
 13 &
 \begin{tikzpicture}[baseline=-.65ex,scale=0.5]
\draw[step=1cm,gray,very thin] (-1,-1) grid (2,1);
\draw[gray, line width=1] (-1,0) -- (2,0);
\filldraw[gray] (-1,0) circle (2pt);
\filldraw[gray] (0,0) circle (2pt);
\filldraw[gray] (1,0) circle (2pt);
\filldraw[gray] (2,0) circle (2pt);
\draw[black, line width=1] (-1,0) -- (0,1) -- (2,1);
\end{tikzpicture} &
 \begin{tikzpicture}[baseline=-.65ex,scale=0.5]
\draw[step=1cm,gray,very thin] (-1,-1) grid (2,1);
\draw[gray, line width=1] (-1,0) -- (2,0);
\filldraw[gray] (-1,0) circle (2pt);
\filldraw[gray] (0,0) circle (2pt);
\filldraw[gray] (1,0) circle (2pt);
\filldraw[gray] (2,0) circle (2pt);
\draw[black, line width=1] (-1,0) -- (1,1) -- (2,1);
\end{tikzpicture} & 
 \begin{tikzpicture}[baseline=-.65ex,scale=0.5]
\draw[step=1cm,gray,very thin] (-1,-1) grid (2,1);
\draw[gray, line width=1] (-1,0) -- (2,0);
\filldraw[gray] (-1,0) circle (2pt);
\filldraw[gray] (0,0) circle (2pt);
\filldraw[gray] (1,0) circle (2pt);
\filldraw[gray] (2,0) circle (2pt);
\draw[black, line width=1] (-1,1) -- (2,-1);
\end{tikzpicture} \\
 \hline
 14 &
 \begin{tikzpicture}[baseline=-.65ex,scale=0.5]
\draw[step=1cm,gray,very thin] (-1,-1) grid (2,1);
\draw[gray, line width=1] (-1,0) -- (2,0);
\filldraw[gray] (-1,0) circle (2pt);
\filldraw[gray] (0,0) circle (2pt);
\filldraw[gray] (1,0) circle (2pt);
\filldraw[gray] (2,0) circle (2pt);
\draw[black, line width=1] (-1,0) -- (0,1) -- (2,1);
\end{tikzpicture} &
 \begin{tikzpicture}[baseline=-.65ex,scale=0.5]
\draw[step=1cm,gray,very thin] (-1,-1) grid (2,1);
\draw[gray, line width=1] (-1,0) -- (2,0);
\filldraw[gray] (-1,0) circle (2pt);
\filldraw[gray] (0,0) circle (2pt);
\filldraw[gray] (1,0) circle (2pt);
\filldraw[gray] (2,0) circle (2pt);
\draw[black, line width=1] (-1,0) -- (2,1);
\end{tikzpicture} & 
 \begin{tikzpicture}[baseline=-.65ex,scale=0.5]
\draw[step=1cm,gray,very thin] (-1,-1) grid (2,1);
\draw[gray, line width=1] (-1,0) -- (2,0);
\filldraw[gray] (-1,0) circle (2pt);
\filldraw[gray] (0,0) circle (2pt);
\filldraw[gray] (1,0) circle (2pt);
\filldraw[gray] (2,0) circle (2pt);
\draw[black, line width=1] (-1,1) -- (2,-1);
\end{tikzpicture} \\
 \hline
 15 &
 \begin{tikzpicture}[baseline=-.65ex,scale=0.5]
\draw[step=1cm,gray,very thin] (-1,-1) grid (2,1);
\draw[gray, line width=1] (-1,0) -- (2,0);
\filldraw[gray] (-1,0) circle (2pt);
\filldraw[gray] (0,0) circle (2pt);
\filldraw[gray] (1,0) circle (2pt);
\filldraw[gray] (2,0) circle (2pt);
\draw[black, line width=1] (-1,0) -- (1,1) -- (2,1);
\end{tikzpicture} &
 \begin{tikzpicture}[baseline=-.65ex,scale=0.5]
\draw[step=1cm,gray,very thin] (-1,-1) grid (2,1);
\draw[gray, line width=1] (-1,0) -- (2,0);
\filldraw[gray] (-1,0) circle (2pt);
\filldraw[gray] (0,0) circle (2pt);
\filldraw[gray] (1,0) circle (2pt);
\filldraw[gray] (2,0) circle (2pt);
\draw[black, line width=1] (-1,0) -- (1,1) -- (2,1);
\end{tikzpicture} & 
 \begin{tikzpicture}[baseline=-.65ex,scale=0.5]
\draw[step=1cm,gray,very thin] (-1,-1) grid (2,1);
\draw[gray, line width=1] (-1,0) -- (2,0);
\filldraw[gray] (-1,0) circle (2pt);
\filldraw[gray] (0,0) circle (2pt);
\filldraw[gray] (1,0) circle (2pt);
\filldraw[gray] (2,0) circle (2pt);
\draw[black, line width=1] (-1,1) -- (2,-1);
\end{tikzpicture} \\
 \hline
 16 &
 \begin{tikzpicture}[baseline=-.65ex,scale=0.5]
\draw[step=1cm,gray,very thin] (-1,-1) grid (2,1);
\draw[gray, line width=1] (-1,0) -- (2,0);
\filldraw[gray] (-1,0) circle (2pt);
\filldraw[gray] (0,0) circle (2pt);
\filldraw[gray] (1,0) circle (2pt);
\filldraw[gray] (2,0) circle (2pt);
\draw[black, line width=1] (-1,0) -- (1,1) -- (2,1);
\end{tikzpicture} &
 \begin{tikzpicture}[baseline=-.65ex,scale=0.5]
\draw[step=1cm,gray,very thin] (-1,-1) grid (2,1);
\draw[gray, line width=1] (-1,0) -- (2,0);
\filldraw[gray] (-1,0) circle (2pt);
\filldraw[gray] (0,0) circle (2pt);
\filldraw[gray] (1,0) circle (2pt);
\filldraw[gray] (2,0) circle (2pt);
\draw[black, line width=1] (-1,0) -- (2,1);
\end{tikzpicture} & 
 \begin{tikzpicture}[baseline=-.65ex,scale=0.5]
\draw[step=1cm,gray,very thin] (-1,-1) grid (2,1);
\draw[gray, line width=1] (-1,0) -- (2,0);
\filldraw[gray] (-1,0) circle (2pt);
\filldraw[gray] (0,0) circle (2pt);
\filldraw[gray] (1,0) circle (2pt);
\filldraw[gray] (2,0) circle (2pt);
\draw[black, line width=1] (-1,1) -- (2,-1);
\end{tikzpicture} \\
 \hline
 17 &
 \begin{tikzpicture}[baseline=-.65ex,scale=0.5]
\draw[step=1cm,gray,very thin] (-1,-1) grid (2,1);
\draw[gray, line width=1] (-1,0) -- (2,0);
\filldraw[gray] (-1,0) circle (2pt);
\filldraw[gray] (0,0) circle (2pt);
\filldraw[gray] (1,0) circle (2pt);
\filldraw[gray] (2,0) circle (2pt);
\draw[black, line width=1] (-1,0) -- (0,1) -- (2,1);
\end{tikzpicture} &
 \begin{tikzpicture}[baseline=-.65ex,scale=0.5]
\draw[step=1cm,gray,very thin] (-1,-1) grid (2,1);
\draw[gray, line width=1] (-1,0) -- (2,0);
\filldraw[gray] (-1,0) circle (2pt);
\filldraw[gray] (0,0) circle (2pt);
\filldraw[gray] (1,0) circle (2pt);
\filldraw[gray] (2,0) circle (2pt);
\draw[black, line width=1] (-1,0) -- (0,1) -- (2,1);
\end{tikzpicture} & 
 \begin{tikzpicture}[baseline=-.65ex,scale=0.5]
\draw[step=1cm,gray,very thin] (-1,-1) grid (2,1);
\draw[gray, line width=1] (-1,0) -- (2,0);
\filldraw[gray] (-1,0) circle (2pt);
\filldraw[gray] (0,0) circle (2pt);
\filldraw[gray] (1,0) circle (2pt);
\filldraw[gray] (2,0) circle (2pt);
\draw[black, line width=1] (-1,1) -- (0,1) -- (2,-1);
\end{tikzpicture} \\
 \hline
 18 &
 \begin{tikzpicture}[baseline=-.65ex,scale=0.5]
\draw[step=1cm,gray,very thin] (-1,-1) grid (2,1);
\draw[gray, line width=1] (-1,0) -- (2,0);
\filldraw[gray] (-1,0) circle (2pt);
\filldraw[gray] (0,0) circle (2pt);
\filldraw[gray] (1,0) circle (2pt);
\filldraw[gray] (2,0) circle (2pt);
\draw[black, line width=1] (-1,0) -- (0,1) -- (2,1);
\end{tikzpicture} &
 \begin{tikzpicture}[baseline=-.65ex,scale=0.5]
\draw[step=1cm,gray,very thin] (-1,-1) grid (2,1);
\draw[gray, line width=1] (-1,0) -- (2,0);
\filldraw[gray] (-1,0) circle (2pt);
\filldraw[gray] (0,0) circle (2pt);
\filldraw[gray] (1,0) circle (2pt);
\filldraw[gray] (2,0) circle (2pt);
\draw[black, line width=1] (-1,0) -- (1,1) -- (2,1);
\end{tikzpicture} & 
 \begin{tikzpicture}[baseline=-.65ex,scale=0.5]
\draw[step=1cm,gray,very thin] (-1,-1) grid (2,1);
\draw[gray, line width=1] (-1,0) -- (2,0);
\filldraw[gray] (-1,0) circle (2pt);
\filldraw[gray] (0,0) circle (2pt);
\filldraw[gray] (1,0) circle (2pt);
\filldraw[gray] (2,0) circle (2pt);
\draw[black, line width=1] (-1,1) -- (0,1) -- (2,-1);
\end{tikzpicture} \\
 \hline
 19 &
 \begin{tikzpicture}[baseline=-.65ex,scale=0.5]
\draw[step=1cm,gray,very thin] (-1,-1) grid (2,1);
\draw[gray, line width=1] (-1,0) -- (2,0);
\filldraw[gray] (-1,0) circle (2pt);
\filldraw[gray] (0,0) circle (2pt);
\filldraw[gray] (1,0) circle (2pt);
\filldraw[gray] (2,0) circle (2pt);
\draw[black, line width=1] (-1,0) -- (1,1) -- (2,1);
\end{tikzpicture} &
 \begin{tikzpicture}[baseline=-.65ex,scale=0.5]
\draw[step=1cm,gray,very thin] (-1,-1) grid (2,1);
\draw[gray, line width=1] (-1,0) -- (2,0);
\filldraw[gray] (-1,0) circle (2pt);
\filldraw[gray] (0,0) circle (2pt);
\filldraw[gray] (1,0) circle (2pt);
\filldraw[gray] (2,0) circle (2pt);
\draw[black, line width=1] (-1,0) -- (1,1) -- (2,1);
\end{tikzpicture} & 
 \begin{tikzpicture}[baseline=-.65ex,scale=0.5]
\draw[step=1cm,gray,very thin] (-1,-1) grid (2,1);
\draw[gray, line width=1] (-1,0) -- (2,0);
\filldraw[gray] (-1,0) circle (2pt);
\filldraw[gray] (0,0) circle (2pt);
\filldraw[gray] (1,0) circle (2pt);
\filldraw[gray] (2,0) circle (2pt);
\draw[black, line width=1] (-1,1) -- (0,1) -- (2,-1);
\end{tikzpicture} \\
 \hline
 20 &
 \begin{tikzpicture}[baseline=-.65ex,scale=0.5]
\draw[step=1cm,gray,very thin] (-1,-1) grid (2,1);
\draw[gray, line width=1] (-1,0) -- (2,0);
\filldraw[gray] (-1,0) circle (2pt);
\filldraw[gray] (0,0) circle (2pt);
\filldraw[gray] (1,0) circle (2pt);
\filldraw[gray] (2,0) circle (2pt);
\draw[black, line width=1] (-1,0) -- (1,1) -- (2,1);
\end{tikzpicture} &
 \begin{tikzpicture}[baseline=-.65ex,scale=0.5]
\draw[step=1cm,gray,very thin] (-1,-1) grid (2,1);
\draw[gray, line width=1] (-1,0) -- (2,0);
\filldraw[gray] (-1,0) circle (2pt);
\filldraw[gray] (0,0) circle (2pt);
\filldraw[gray] (1,0) circle (2pt);
\filldraw[gray] (2,0) circle (2pt);
\draw[black, line width=1] (-1,0) -- (1,1) -- (2,1);
\end{tikzpicture} & 
 \begin{tikzpicture}[baseline=-.65ex,scale=0.5]
\draw[step=1cm,gray,very thin] (-1,-1) grid (2,1);
\draw[gray, line width=1] (-1,0) -- (2,0);
\filldraw[gray] (-1,0) circle (2pt);
\filldraw[gray] (0,0) circle (2pt);
\filldraw[gray] (1,0) circle (2pt);
\filldraw[gray] (2,0) circle (2pt);
\draw[black, line width=1] (-1,1) -- (1,0) -- (2,-1);
\end{tikzpicture}

\end{tabular}
\caption[]{Base polytope length $m=3$, number of functions $n=3$}
\label{table:en3}
\end{table}

\begin{table}
\centering
\begin{tabular}{c|c|c|c}
Number & $\Psi_1'$ & $\Psi_2'$ & $\Psi_3'$ \\
\hline
21 &
\begin{tikzpicture}[baseline=-.65ex,scale=0.5]
\draw[step=1cm,gray,very thin] (-1,-1) grid (3,1);
\draw[gray, line width=1] (-1,0) -- (3,0);
\filldraw[gray] (-1,0) circle (2pt);
\filldraw[gray] (0,0) circle (2pt);
\filldraw[gray] (1,0) circle (2pt);
\filldraw[gray] (2,0) circle (2pt);
\filldraw[gray] (3,0) circle (2pt);
\draw[black, line width=1] (-1,0) -- (0,1) -- (3,1);
\end{tikzpicture} &
\begin{tikzpicture}[baseline=-.65ex,scale=0.5]
\draw[step=1cm,gray,very thin] (-1,-1) grid (3,1);
\draw[gray, line width=1] (-1,0) -- (3,0);
\filldraw[gray] (-1,0) circle (2pt);
\filldraw[gray] (0,0) circle (2pt);
\filldraw[gray] (1,0) circle (2pt);
\filldraw[gray] (2,0) circle (2pt);
\filldraw[gray] (3,0) circle (2pt);
\draw[black, line width=1] (-1,0) -- (0,1) -- (3,1);
\end{tikzpicture} &
\begin{tikzpicture}[baseline=-.65ex,scale=0.5]
\draw[step=1cm,gray,very thin] (-1,-1) grid (3,1);
\draw[gray, line width=1] (-1,0) -- (3,0);
\filldraw[gray] (-1,0) circle (2pt);
\filldraw[gray] (0,0) circle (2pt);
\filldraw[gray] (1,0) circle (2pt);
\filldraw[gray] (2,0) circle (2pt);
\filldraw[gray] (3,0) circle (2pt);
\draw[black, line width=1] (-1,1) -- (3,-1);
\end{tikzpicture} \\
\hline
22 &
\begin{tikzpicture}[baseline=-.65ex,scale=0.5]
\draw[step=1cm,gray,very thin] (-1,-1) grid (3,1);
\draw[gray, line width=1] (-1,0) -- (3,0);
\filldraw[gray] (-1,0) circle (2pt);
\filldraw[gray] (0,0) circle (2pt);
\filldraw[gray] (1,0) circle (2pt);
\filldraw[gray] (2,0) circle (2pt);
\filldraw[gray] (3,0) circle (2pt);
\draw[black, line width=1] (-1,0) -- (0,1) -- (3,1);
\end{tikzpicture} &
\begin{tikzpicture}[baseline=-.65ex,scale=0.5]
\draw[step=1cm,gray,very thin] (-1,-1) grid (3,1);
\draw[gray, line width=1] (-1,0) -- (3,0);
\filldraw[gray] (-1,0) circle (2pt);
\filldraw[gray] (0,0) circle (2pt);
\filldraw[gray] (1,0) circle (2pt);
\filldraw[gray] (2,0) circle (2pt);
\filldraw[gray] (3,0) circle (2pt);
\draw[black, line width=1] (-1,0) -- (1,1) -- (3,1);
\end{tikzpicture} &
\begin{tikzpicture}[baseline=-.65ex,scale=0.5]
\draw[step=1cm,gray,very thin] (-1,-1) grid (3,1);
\draw[gray, line width=1] (-1,0) -- (3,0);
\filldraw[gray] (-1,0) circle (2pt);
\filldraw[gray] (0,0) circle (2pt);
\filldraw[gray] (1,0) circle (2pt);
\filldraw[gray] (2,0) circle (2pt);
\filldraw[gray] (3,0) circle (2pt);
\draw[black, line width=1] (-1,1) -- (3,-1);
\end{tikzpicture} \\
\hline
23&
\begin{tikzpicture}[baseline=-.65ex,scale=0.5]
\draw[step=1cm,gray,very thin] (-1,-1) grid (3,1);
\draw[gray, line width=1] (-1,0) -- (3,0);
\filldraw[gray] (-1,0) circle (2pt);
\filldraw[gray] (0,0) circle (2pt);
\filldraw[gray] (1,0) circle (2pt);
\filldraw[gray] (2,0) circle (2pt);
\filldraw[gray] (3,0) circle (2pt);
\draw[black, line width=1] (-1,0) -- (0,1) -- (3,1);
\end{tikzpicture} &
\begin{tikzpicture}[baseline=-.65ex,scale=0.5]
\draw[step=1cm,gray,very thin] (-1,-1) grid (3,1);
\draw[gray, line width=1] (-1,0) -- (3,0);
\filldraw[gray] (-1,0) circle (2pt);
\filldraw[gray] (0,0) circle (2pt);
\filldraw[gray] (1,0) circle (2pt);
\filldraw[gray] (2,0) circle (2pt);
\filldraw[gray] (3,0) circle (2pt);
\draw[black, line width=1] (-1,0) -- (2,1) -- (3,1);
\end{tikzpicture} &
\begin{tikzpicture}[baseline=-.65ex,scale=0.5]
\draw[step=1cm,gray,very thin] (-1,-1) grid (3,1);
\draw[gray, line width=1] (-1,0) -- (3,0);
\filldraw[gray] (-1,0) circle (2pt);
\filldraw[gray] (0,0) circle (2pt);
\filldraw[gray] (1,0) circle (2pt);
\filldraw[gray] (2,0) circle (2pt);
\filldraw[gray] (3,0) circle (2pt);
\draw[black, line width=1] (-1,1) -- (3,-1);
\end{tikzpicture} \\
\hline
24 &
\begin{tikzpicture}[baseline=-.65ex,scale=0.5]
\draw[step=1cm,gray,very thin] (-1,-1) grid (3,1);
\draw[gray, line width=1] (-1,0) -- (3,0);
\filldraw[gray] (-1,0) circle (2pt);
\filldraw[gray] (0,0) circle (2pt);
\filldraw[gray] (1,0) circle (2pt);
\filldraw[gray] (2,0) circle (2pt);
\filldraw[gray] (3,0) circle (2pt);
\draw[black, line width=1] (-1,0) -- (0,1) -- (3,1);
\end{tikzpicture} &
\begin{tikzpicture}[baseline=-.65ex,scale=0.5]
\draw[step=1cm,gray,very thin] (-1,-1) grid (3,1);
\draw[gray, line width=1] (-1,0) -- (3,0);
\filldraw[gray] (-1,0) circle (2pt);
\filldraw[gray] (0,0) circle (2pt);
\filldraw[gray] (1,0) circle (2pt);
\filldraw[gray] (2,0) circle (2pt);
\filldraw[gray] (3,0) circle (2pt);
\draw[black, line width=1] (-1,0) -- (3,1);
\end{tikzpicture} &
\begin{tikzpicture}[baseline=-.65ex,scale=0.5]
\draw[step=1cm,gray,very thin] (-1,-1) grid (3,1);
\draw[gray, line width=1] (-1,0) -- (3,0);
\filldraw[gray] (-1,0) circle (2pt);
\filldraw[gray] (0,0) circle (2pt);
\filldraw[gray] (1,0) circle (2pt);
\filldraw[gray] (2,0) circle (2pt);
\filldraw[gray] (3,0) circle (2pt);
\draw[black, line width=1] (-1,1) -- (3,-1);
\end{tikzpicture} \\

\hline
25 &
\begin{tikzpicture}[baseline=-.65ex,scale=0.5]
\draw[step=1cm,gray,very thin] (-1,-1) grid (3,1);
\draw[gray, line width=1] (-1,0) -- (3,0);
\filldraw[gray] (-1,0) circle (2pt);
\filldraw[gray] (0,0) circle (2pt);
\filldraw[gray] (1,0) circle (2pt);
\filldraw[gray] (2,0) circle (2pt);
\filldraw[gray] (3,0) circle (2pt);
\draw[black, line width=1] (-1,0) -- (1,1) -- (3,1);
\end{tikzpicture} &
\begin{tikzpicture}[baseline=-.65ex,scale=0.5]
\draw[step=1cm,gray,very thin] (-1,-1) grid (3,1);
\draw[gray, line width=1] (-1,0) -- (3,0);
\filldraw[gray] (-1,0) circle (2pt);
\filldraw[gray] (0,0) circle (2pt);
\filldraw[gray] (1,0) circle (2pt);
\filldraw[gray] (2,0) circle (2pt);
\filldraw[gray] (3,0) circle (2pt);
\draw[black, line width=1] (-1,0) -- (1,1) -- (3,1);
\end{tikzpicture} &
\begin{tikzpicture}[baseline=-.65ex,scale=0.5]
\draw[step=1cm,gray,very thin] (-1,-1) grid (3,1);
\draw[gray, line width=1] (-1,0) -- (3,0);
\filldraw[gray] (-1,0) circle (2pt);
\filldraw[gray] (0,0) circle (2pt);
\filldraw[gray] (1,0) circle (2pt);
\filldraw[gray] (2,0) circle (2pt);
\filldraw[gray] (3,0) circle (2pt);
\draw[black, line width=1] (-1,1) -- (3,-1);
\end{tikzpicture} \\
\hline
26 &
\begin{tikzpicture}[baseline=-.65ex,scale=0.5]
\draw[step=1cm,gray,very thin] (-1,-1) grid (3,1);
\draw[gray, line width=1] (-1,0) -- (3,0);
\filldraw[gray] (-1,0) circle (2pt);
\filldraw[gray] (0,0) circle (2pt);
\filldraw[gray] (1,0) circle (2pt);
\filldraw[gray] (2,0) circle (2pt);
\filldraw[gray] (3,0) circle (2pt);
\draw[black, line width=1] (-1,0) -- (1,1) -- (3,1);
\end{tikzpicture} &
\begin{tikzpicture}[baseline=-.65ex,scale=0.5]
\draw[step=1cm,gray,very thin] (-1,-1) grid (3,1);
\draw[gray, line width=1] (-1,0) -- (3,0);
\filldraw[gray] (-1,0) circle (2pt);
\filldraw[gray] (0,0) circle (2pt);
\filldraw[gray] (1,0) circle (2pt);
\filldraw[gray] (2,0) circle (2pt);
\filldraw[gray] (3,0) circle (2pt);
\draw[black, line width=1] (-1,0) -- (2,1) -- (3,1);
\end{tikzpicture} &
\begin{tikzpicture}[baseline=-.65ex,scale=0.5]
\draw[step=1cm,gray,very thin] (-1,-1) grid (3,1);
\draw[gray, line width=1] (-1,0) -- (3,0);
\filldraw[gray] (-1,0) circle (2pt);
\filldraw[gray] (0,0) circle (2pt);
\filldraw[gray] (1,0) circle (2pt);
\filldraw[gray] (2,0) circle (2pt);
\filldraw[gray] (3,0) circle (2pt);
\draw[black, line width=1] (-1,1) -- (3,-1);
\end{tikzpicture} \\
\hline
27 &
\begin{tikzpicture}[baseline=-.65ex,scale=0.5]
\draw[step=1cm,gray,very thin] (-1,-1) grid (3,1);
\draw[gray, line width=1] (-1,0) -- (3,0);
\filldraw[gray] (-1,0) circle (2pt);
\filldraw[gray] (0,0) circle (2pt);
\filldraw[gray] (1,0) circle (2pt);
\filldraw[gray] (2,0) circle (2pt);
\filldraw[gray] (3,0) circle (2pt);
\draw[black, line width=1] (-1,0) -- (1,1) -- (3,1);
\end{tikzpicture} &
\begin{tikzpicture}[baseline=-.65ex,scale=0.5]
\draw[step=1cm,gray,very thin] (-1,-1) grid (3,1);
\draw[gray, line width=1] (-1,0) -- (3,0);
\filldraw[gray] (-1,0) circle (2pt);
\filldraw[gray] (0,0) circle (2pt);
\filldraw[gray] (1,0) circle (2pt);
\filldraw[gray] (2,0) circle (2pt);
\filldraw[gray] (3,0) circle (2pt);
\draw[black, line width=1] (-1,0) -- (3,1);
\end{tikzpicture} &
\begin{tikzpicture}[baseline=-.65ex,scale=0.5]
\draw[step=1cm,gray,very thin] (-1,-1) grid (3,1);
\draw[gray, line width=1] (-1,0) -- (3,0);
\filldraw[gray] (-1,0) circle (2pt);
\filldraw[gray] (0,0) circle (2pt);
\filldraw[gray] (1,0) circle (2pt);
\filldraw[gray] (2,0) circle (2pt);
\filldraw[gray] (3,0) circle (2pt);
\draw[black, line width=1] (-1,1) -- (3,-1);
\end{tikzpicture} \\
\hline
28 &
\begin{tikzpicture}[baseline=-.65ex,scale=0.5]
\draw[step=1cm,gray,very thin] (-1,-1) grid (3,1);
\draw[gray, line width=1] (-1,0) -- (3,0);
\filldraw[gray] (-1,0) circle (2pt);
\filldraw[gray] (0,0) circle (2pt);
\filldraw[gray] (1,0) circle (2pt);
\filldraw[gray] (2,0) circle (2pt);
\filldraw[gray] (3,0) circle (2pt);
\draw[black, line width=1] (-1,0) -- (2,1) -- (3,1);
\end{tikzpicture} &
\begin{tikzpicture}[baseline=-.65ex,scale=0.5]
\draw[step=1cm,gray,very thin] (-1,-1) grid (3,1);
\draw[gray, line width=1] (-1,0) -- (3,0);
\filldraw[gray] (-1,0) circle (2pt);
\filldraw[gray] (0,0) circle (2pt);
\filldraw[gray] (1,0) circle (2pt);
\filldraw[gray] (2,0) circle (2pt);
\filldraw[gray] (3,0) circle (2pt);
\draw[black, line width=1] (-1,0) -- (2,1) -- (3,1);
\end{tikzpicture} &
\begin{tikzpicture}[baseline=-.65ex,scale=0.5]
\draw[step=1cm,gray,very thin] (-1,-1) grid (3,1);
\draw[gray, line width=1] (-1,0) -- (3,0);
\filldraw[gray] (-1,0) circle (2pt);
\filldraw[gray] (0,0) circle (2pt);
\filldraw[gray] (1,0) circle (2pt);
\filldraw[gray] (2,0) circle (2pt);
\filldraw[gray] (3,0) circle (2pt);
\draw[black, line width=1] (-1,1) -- (3,-1);
\end{tikzpicture} \\
\hline
29 &
\begin{tikzpicture}[baseline=-.65ex,scale=0.5]
\draw[step=1cm,gray,very thin] (-1,-1) grid (3,1);
\draw[gray, line width=1] (-1,0) -- (3,0);
\filldraw[gray] (-1,0) circle (2pt);
\filldraw[gray] (0,0) circle (2pt);
\filldraw[gray] (1,0) circle (2pt);
\filldraw[gray] (2,0) circle (2pt);
\filldraw[gray] (3,0) circle (2pt);
\draw[black, line width=1] (-1,0) -- (2,1) -- (3,1);
\end{tikzpicture} &
\begin{tikzpicture}[baseline=-.65ex,scale=0.5]
\draw[step=1cm,gray,very thin] (-1,-1) grid (3,1);
\draw[gray, line width=1] (-1,0) -- (3,0);
\filldraw[gray] (-1,0) circle (2pt);
\filldraw[gray] (0,0) circle (2pt);
\filldraw[gray] (1,0) circle (2pt);
\filldraw[gray] (2,0) circle (2pt);
\filldraw[gray] (3,0) circle (2pt);
\draw[black, line width=1] (-1,0) -- (3,1);
\end{tikzpicture} &
\begin{tikzpicture}[baseline=-.65ex,scale=0.5]
\draw[step=1cm,gray,very thin] (-1,-1) grid (3,1);
\draw[gray, line width=1] (-1,0) -- (3,0);
\filldraw[gray] (-1,0) circle (2pt);
\filldraw[gray] (0,0) circle (2pt);
\filldraw[gray] (1,0) circle (2pt);
\filldraw[gray] (2,0) circle (2pt);
\filldraw[gray] (3,0) circle (2pt);
\draw[black, line width=1] (-1,1) -- (3,-1);
\end{tikzpicture} \\
\end{tabular}
\caption[]{Base polytope length $m=4$, number of functions $n=3$}
\label{table:en4}
\end{table}

\begin{table}
\centering
\begin{tabular}{c|c|c|c}
Number & $\Psi_1'$ & $\Psi_2'$ & $\Psi_3'$ \\
\hline
30 & 
\begin{tikzpicture}[baseline=-.65ex,scale=0.4]
\draw[step=1cm,gray,very thin] (-1,-2) grid (5,2);
\draw[gray, line width=1] (-1,0) -- (5,0);
\filldraw[gray] (-1,0) circle (2pt);
\filldraw[gray] (0,0) circle (2pt);
\filldraw[gray] (1,0) circle (2pt);
\filldraw[gray] (2,0) circle (2pt);
\filldraw[gray] (3,0) circle (2pt);
\filldraw[gray] (4,0) circle (2pt);
\filldraw[gray] (5,0) circle (2pt);
\draw[black, line width=1] (-1,0) -- (0,1) -- (5,1);
\end{tikzpicture} &
\begin{tikzpicture}[baseline=-.65ex,scale=0.4]
\draw[step=1cm,gray,very thin] (-1,-2) grid (5,2);
\draw[gray, line width=1] (-1,0) -- (5,0);
\filldraw[gray] (-1,0) circle (2pt);
\filldraw[gray] (0,0) circle (2pt);
\filldraw[gray] (1,0) circle (2pt);
\filldraw[gray] (2,0) circle (2pt);
\filldraw[gray] (3,0) circle (2pt);
\filldraw[gray] (4,0) circle (2pt);
\filldraw[gray] (5,0) circle (2pt);
\draw[black, line width=1] (-1,0) -- (5,2);
\end{tikzpicture} &
\begin{tikzpicture}[baseline=-.65ex,scale=0.4]
\draw[step=1cm,gray,very thin] (-1,-2) grid (5,2);
\draw[gray, line width=1] (-1,0) -- (5,0);
\filldraw[gray] (-1,0) circle (2pt);
\filldraw[gray] (0,0) circle (2pt);
\filldraw[gray] (1,0) circle (2pt);
\filldraw[gray] (2,0) circle (2pt);
\filldraw[gray] (3,0) circle (2pt);
\filldraw[gray] (4,0) circle (2pt);
\filldraw[gray] (5,0) circle (2pt);
\draw[black, line width=1] (-1,1) -- (5,-2);
\end{tikzpicture} \\
\hline
31 & 
\begin{tikzpicture}[baseline=-.65ex,scale=0.4]
\draw[step=1cm,gray,very thin] (-1,-2) grid (5,2);
\draw[gray, line width=1] (-1,0) -- (5,0);
\filldraw[gray] (-1,0) circle (2pt);
\filldraw[gray] (0,0) circle (2pt);
\filldraw[gray] (1,0) circle (2pt);
\filldraw[gray] (2,0) circle (2pt);
\filldraw[gray] (3,0) circle (2pt);
\filldraw[gray] (4,0) circle (2pt);
\filldraw[gray] (5,0) circle (2pt);
\draw[black, line width=1] (-1,0) -- (1,1) -- (5,1);
\end{tikzpicture} &
\begin{tikzpicture}[baseline=-.65ex,scale=0.4]
\draw[step=1cm,gray,very thin] (-1,-2) grid (5,2);
\draw[gray, line width=1] (-1,0) -- (5,0);
\filldraw[gray] (-1,0) circle (2pt);
\filldraw[gray] (0,0) circle (2pt);
\filldraw[gray] (1,0) circle (2pt);
\filldraw[gray] (2,0) circle (2pt);
\filldraw[gray] (3,0) circle (2pt);
\filldraw[gray] (4,0) circle (2pt);
\filldraw[gray] (5,0) circle (2pt);
\draw[black, line width=1] (-1,0) -- (5,2);
\end{tikzpicture} &
\begin{tikzpicture}[baseline=-.65ex,scale=0.4]
\draw[step=1cm,gray,very thin] (-1,-2) grid (5,2);
\draw[gray, line width=1] (-1,0) -- (5,0);
\filldraw[gray] (-1,0) circle (2pt);
\filldraw[gray] (0,0) circle (2pt);
\filldraw[gray] (1,0) circle (2pt);
\filldraw[gray] (2,0) circle (2pt);
\filldraw[gray] (3,0) circle (2pt);
\filldraw[gray] (4,0) circle (2pt);
\filldraw[gray] (5,0) circle (2pt);
\draw[black, line width=1] (-1,1) -- (5,-2);
\end{tikzpicture} \\
\hline
32 & 
\begin{tikzpicture}[baseline=-.65ex,scale=0.4]
\draw[step=1cm,gray,very thin] (-1,-2) grid (5,2);
\draw[gray, line width=1] (-1,0) -- (5,0);
\filldraw[gray] (-1,0) circle (2pt);
\filldraw[gray] (0,0) circle (2pt);
\filldraw[gray] (1,0) circle (2pt);
\filldraw[gray] (2,0) circle (2pt);
\filldraw[gray] (3,0) circle (2pt);
\filldraw[gray] (4,0) circle (2pt);
\filldraw[gray] (5,0) circle (2pt);
\draw[black, line width=1] (-1,0) -- (2,1) -- (5,1);
\end{tikzpicture} &
\begin{tikzpicture}[baseline=-.65ex,scale=0.4]
\draw[step=1cm,gray,very thin] (-1,-2) grid (5,2);
\draw[gray, line width=1] (-1,0) -- (5,0);
\filldraw[gray] (-1,0) circle (2pt);
\filldraw[gray] (0,0) circle (2pt);
\filldraw[gray] (1,0) circle (2pt);
\filldraw[gray] (2,0) circle (2pt);
\filldraw[gray] (3,0) circle (2pt);
\filldraw[gray] (4,0) circle (2pt);
\filldraw[gray] (5,0) circle (2pt);
\draw[black, line width=1] (-1,0) -- (5,2);
\end{tikzpicture} &
\begin{tikzpicture}[baseline=-.65ex,scale=0.4]
\draw[step=1cm,gray,very thin] (-1,-2) grid (5,2);
\draw[gray, line width=1] (-1,0) -- (5,0);
\filldraw[gray] (-1,0) circle (2pt);
\filldraw[gray] (0,0) circle (2pt);
\filldraw[gray] (1,0) circle (2pt);
\filldraw[gray] (2,0) circle (2pt);
\filldraw[gray] (3,0) circle (2pt);
\filldraw[gray] (4,0) circle (2pt);
\filldraw[gray] (5,0) circle (2pt);
\draw[black, line width=1] (-1,1) -- (5,-2);
\end{tikzpicture} \\
\hline
33 & 
\begin{tikzpicture}[baseline=-.65ex,scale=0.4]
\draw[step=1cm,gray,very thin] (-1,-2) grid (5,2);
\draw[gray, line width=1] (-1,0) -- (5,0);
\filldraw[gray] (-1,0) circle (2pt);
\filldraw[gray] (0,0) circle (2pt);
\filldraw[gray] (1,0) circle (2pt);
\filldraw[gray] (2,0) circle (2pt);
\filldraw[gray] (3,0) circle (2pt);
\filldraw[gray] (4,0) circle (2pt);
\filldraw[gray] (5,0) circle (2pt);
\draw[black, line width=1] (-1,0) -- (3,1) -- (5,1);
\end{tikzpicture} &
\begin{tikzpicture}[baseline=-.65ex,scale=0.4]
\draw[step=1cm,gray,very thin] (-1,-2) grid (5,2);
\draw[gray, line width=1] (-1,0) -- (5,0);
\filldraw[gray] (-1,0) circle (2pt);
\filldraw[gray] (0,0) circle (2pt);
\filldraw[gray] (1,0) circle (2pt);
\filldraw[gray] (2,0) circle (2pt);
\filldraw[gray] (3,0) circle (2pt);
\filldraw[gray] (4,0) circle (2pt);
\filldraw[gray] (5,0) circle (2pt);
\draw[black, line width=1] (-1,0) -- (5,2);
\end{tikzpicture} &
\begin{tikzpicture}[baseline=-.65ex,scale=0.4]
\draw[step=1cm,gray,very thin] (-1,-2) grid (5,2);
\draw[gray, line width=1] (-1,0) -- (5,0);
\filldraw[gray] (-1,0) circle (2pt);
\filldraw[gray] (0,0) circle (2pt);
\filldraw[gray] (1,0) circle (2pt);
\filldraw[gray] (2,0) circle (2pt);
\filldraw[gray] (3,0) circle (2pt);
\filldraw[gray] (4,0) circle (2pt);
\filldraw[gray] (5,0) circle (2pt);
\draw[black, line width=1] (-1,1) -- (5,-2);
\end{tikzpicture} \\
\hline
34 & 
\begin{tikzpicture}[baseline=-.65ex,scale=0.4]
\draw[step=1cm,gray,very thin] (-1,-2) grid (5,2);
\draw[gray, line width=1] (-1,0) -- (5,0);
\filldraw[gray] (-1,0) circle (2pt);
\filldraw[gray] (0,0) circle (2pt);
\filldraw[gray] (1,0) circle (2pt);
\filldraw[gray] (2,0) circle (2pt);
\filldraw[gray] (3,0) circle (2pt);
\filldraw[gray] (4,0) circle (2pt);
\filldraw[gray] (5,0) circle (2pt);
\draw[black, line width=1] (-1,0) -- (4,1) -- (5,1);
\end{tikzpicture} &
\begin{tikzpicture}[baseline=-.65ex,scale=0.4]
\draw[step=1cm,gray,very thin] (-1,-2) grid (5,2);
\draw[gray, line width=1] (-1,0) -- (5,0);
\filldraw[gray] (-1,0) circle (2pt);
\filldraw[gray] (0,0) circle (2pt);
\filldraw[gray] (1,0) circle (2pt);
\filldraw[gray] (2,0) circle (2pt);
\filldraw[gray] (3,0) circle (2pt);
\filldraw[gray] (4,0) circle (2pt);
\filldraw[gray] (5,0) circle (2pt);
\draw[black, line width=1] (-1,0) -- (5,2);
\end{tikzpicture} &
\begin{tikzpicture}[baseline=-.65ex,scale=0.4]
\draw[step=1cm,gray,very thin] (-1,-2) grid (5,2);
\draw[gray, line width=1] (-1,0) -- (5,0);
\filldraw[gray] (-1,0) circle (2pt);
\filldraw[gray] (0,0) circle (2pt);
\filldraw[gray] (1,0) circle (2pt);
\filldraw[gray] (2,0) circle (2pt);
\filldraw[gray] (3,0) circle (2pt);
\filldraw[gray] (4,0) circle (2pt);
\filldraw[gray] (5,0) circle (2pt);
\draw[black, line width=1] (-1,1) -- (5,-2);
\end{tikzpicture} \\

\end{tabular}
\caption[]{Base polytope length $m=6$, number of functions $n=3$}
\label{table:en6}
\end{table}

\section{Three Dimensional Fano $T$-Varieties}\label{sec:threefolds}
For a fixed polytope, the possible functions in a CDP are quite
restricted by the regions of linearity that the polytope supports, and
conditions of integrality at lattice points. In the three dimension
case these ideas are sufficient to determine a global bound on the
number of functions from Theorem~\ref{bound}.

\begin{thm}\label{thm:3dbound}
Every normalized Fano CDP with two-dimensional base has at most $8$ 
  functions. Furthermore, this bound is sharp in the sense that there exist normalized Fano CDPs with two-dimensional base and $8$ functions.
\end{thm}

Recall, Example~\ref{ex:cross} is a Fano CDP with 8 non-linear functions
supported on the two-dimensional cross polytope. This shows that the bound in the theorem is sharp.

To prove the remainder of the theorem we divide the set of base
polytopes into three cases depending on the lattice basis elements
they contain. The intuition, and some of the tools are best developed
through a few examples.

\begin{ex}\label{ex:tri}
We prove Theorem~\ref{thm:3dbound} in the  case of the triangular
  polytope $\Box$ pictured in  Figure~\ref{fig:tri}. Our convention
  throughout this section   is to indicate the origin by a dot.

Suppose that $\Psi$ is a normalized Fano CDP over $\Box$.  The dotted lines
indicate the unique maximal triangulation of $\Box$ using lattice
points. A region of linearity for any of the functions $\Psi'$
is necessarily a union of neighboring triangles in the triangulation. 

Taking $v$ to be the vector $(1,0)$, we have $\alpha_v=1/2$ and by Lemma \ref{lemma:bound} there are at most $8$
functions $\Psi_i'$ which are either non-integral, or non-linear along
the horizontal axis.  On the other hand, consider any function $\Psi_i'$ which is integral, and linear along the horizontal axis. By Lemma \ref{linearalonglines}, this function is linear along the segment from the origin to $(0,-1)$, which implies that it is linear on the union of the triangles $A$ and $B$ in the figure. By
concavity it is thus linear across all of $\Box$, but since $\Psi$ is normalized, it has no linear integral functions. 

Hence, $\Psi$ has at most $8$ functions.

\end{ex}
\begin{figure}
	\begin{tikzpicture}[scale=.5mm]
\draw[step=1cm,black,very thin] (-1,-1) grid (1,1);
\filldraw[black] (0,0) circle (2pt);
\draw[black, line width=2] (0,-1) -- (1,1) -- (-1,1)--(0,-1);
\draw[gray, dotted, line width=2] (0,0) -- (1,1);
\draw[gray, dotted, line width=2] (0,0) -- (-1,1);
\draw[gray, dotted, line width=2] (0,0) -- (0,1);
\draw[gray, dotted, line width=2] (0,0) -- (0,-1);
\node at (-.3,-.1) {$A$};
\node at (.3,-.1) {$B$};
\end{tikzpicture}
\caption{The base polytope in Example \ref{ex:tri} and its regions of linearity}\label{fig:tri}
\end{figure}
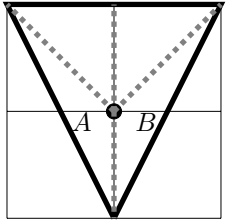

This example can be generalized to the following lemma:

\begin{lemma}\label{lemma:lbound}
Let $\Psi$ be a normalized Fano CDP with two-dimensional base $\Box\subset
M\otimes {\mathbb{R}}$. Consider any $v\in M$ such that the line $\langle v \rangle$ spanned by $v$ intersects the boundary of $\Box$ in a non-lattice point. Then $\Psi$ has at most $4/\alpha_v$ functions, where 
\[
	\alpha_v=\min\left\{1,\max\{\alpha\in \RR_{\geq 0}\ |\ \pm \alpha v \in \Box\}\right\}.
\]
\end{lemma}
\begin{proof}
  By Lemma \ref{lemma:bound}, $\Psi$ has at most $4/\alpha_v$
  functions $\Psi_i'$ which are either non-integral, or non-linear
  along the line $\langle v \rangle$.  Consider instead an integral
  function $\Psi_i'$ which is linear along $\langle v \rangle$. Then
  by concavity and the fact that $\Psi_i'$ is linear along lines from
  the origin (Lemma \ref{linearalonglines}), $\Psi_i'$ has at most two
  regions of linearity, divided by exactly the line
  $\langle v \rangle$.

	We have assumed that this line intersects the boundary of
        $\Box$ in some point $y\notin M$, which thus cannot lie under
        a vertex of the graph of $\Psi_i'$ (since it has integral
        vertices). If $F$ is the facet of $\Box$ containing $y$ in its
        interior, it follows that $\Psi_i'$ must be linear along $F$
        in a neighborhood of $y$.
Since $\Psi_i'$ is also linear along lines from the origin, it is linear in a neighborhood of $y$. 
 But by the descriptions of the regions
        of linearity of $\Psi_i'$ in the preceding paragraph, it
        follows that $\Psi_i'$ is linear on all of $\Box$. But since $\Psi$ is normalized, it has no integral linear functions.
\end{proof}

Next we prove Theorem~\ref{thm:3dbound} for a second type of polytope.

\begin{ex}\label{ex:ep}
  Let $\Box$ be a base polytope which contains the gray polytope
  pictured in
  Figure~\ref{fig:ep} as a (possibly non-proper) subset, where the origin is the marked lattice point.

We claim that any normalized Fano CDP
  $\Psi$ supported on $\Box$ consists of at most $8$ functions.

There are two cases to consider depending on how the dotted line
spanned by $v=(2,1)$ intersects the polytope. If the intersection of
the dotted line spanned by $v$ with the boundary of $\Box$ contains
only lattice points, then $(-1,-2)$ must be in $\Box$ since that is
the lattice point closest to the origin that it will pass through. Taking a
lattice basis given by $(1,2)$ and $(0,1)$, Theorem~\ref{bound} yields
that $\Psi$ has at most $c(\Box)=8$ functions. 

If instead the line intersects the boundary at a non-lattice point, we
again use Lemma~\ref{lemma:lbound} to bound the number of functions by $4/\alpha_v$. By visual inspection, $\alpha_v\geq 1/2$,
hence $4/\alpha_v\leq 8$. 

In either case, the number of non-trivial functions is bounded by 8.
\end{ex}
\begin{figure}
	\begin{tikzpicture}[scale=.5mm]
\draw[gray,fill, line width=2] (1,2) -- (-1,0) -- (-1,-1) -- (0,-1) -- (1,2);
\draw[step=1cm,black,very thin] (-2,-2) grid (1,2);
\filldraw[black] (0,0) circle (2pt);
\draw[black, dotted, line width=2] (-1,-2) -- (1,2);
\end{tikzpicture}
\caption{The base polytope in Example~\ref{ex:ep}}\label{fig:ep}
\end{figure}
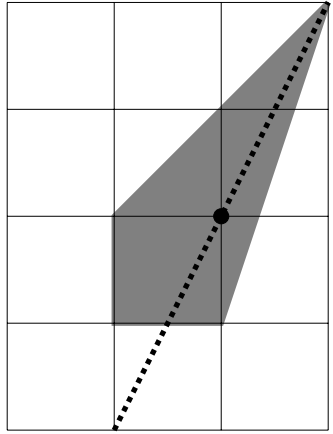

\begin{lemma}\label{lemma:orth}
	Let $\Box$ be a two-dimensional lattice polytope in $M\otimes \RR$ containing the origin in its interior. Then either $\Box$ is equivalent to the two-dimensional cross polytope, or there is a basis $e_1,e_2$ of $M$ such that $-e_1,-e_2\in \Box$, along with some $ae_1+be_2\in \Box$ for $a,b\in \ZZ_{\geq 0}$. 
\end{lemma}
\begin{proof}
	Let $\rho_1,\ldots,\rho_m$ be the rays through all non-zero lattice points of $\Box$, ordered consecutively; let $v_i$ denote the primitive lattice generator of $\rho_i$. Then $v_i$ and $v_{i+1}$ form basis of $M$ for all $i$, see e.g. \cite[\S2.6]{fulton}. Here indices are taken modulo $m$. Furthermore, $v_i\in \Box$ for all $i$.

Now, $-\rho_1$ is contained in some cone $\sigma$ spanned by $\rho_i$ and $\rho_{i+1}$. If $-\rho_i$ is in the interior of this cone, we take $e_1=-v_i$, $e_2=-v_{i+1}$, implying that $v_1=ae_1+be_2$ for $a,b> 0$.

Suppose instead that $-\rho_1$ is on the boundary; without loss of generality $-\rho_1=\rho_i$. If $-\rho_2$ is in the interior of the cone $\sigma$, then as above taking $e_1=-v_i$, $e_2=-v_{i+1}$ implies that $v_2=ae_1+be_2$ for $a,b> 0$. If $-\rho_2$ is in the boundary of $\sigma$, then $-v_2=v_{i+1}$. By assumption we already had $-v_1=v_i$, so $\Box$ contains $\pm e_1,\pm e_2$, that is, the two-dimensional cross polytope. If $\Box$ contains any other lattice points, then clearly the above criterion is satisfied, otherwise $\Box$ is equivalent to the cross polytope.

Finally, if $-\rho_2$ is not in $\sigma$ at all, then $-\rho_{i+1}$ is in the interior of the cone generated by $\rho_1$ and $\rho_2$, so taking $e_1=-v_1$, $e_2=-v_{2}$  implies that $v_{i+1}=ae_1+be_2$ for $a,b> 0$.
\end{proof}

\begin{figure}
	\begin{tikzpicture}[scale=.5mm]
\draw[gray,fill, line width=2] (1,1) -- (-1,0)  -- (0,-1) -- (1,1);
\draw[step=1cm,black,very thin] (-2,-2) grid (2,2);
\filldraw[black] (0,0) circle (2pt);
\draw (1,0) node[cross=4pt,red] {};
\draw (0,1) node[cross=4pt,red] {};
\end{tikzpicture}
\caption{Case \ref{caseduo} in proof of Theorem \ref{thm:3dbound}}\label{fig:caseduo}
\end{figure}
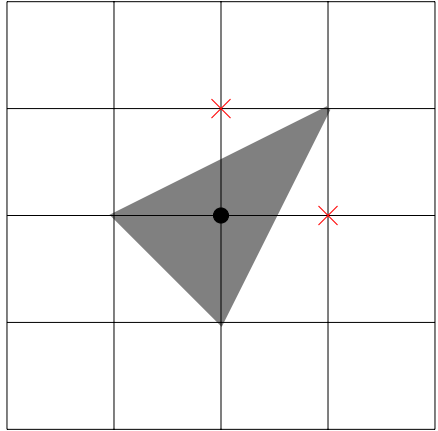
\begin{figure}
	\begin{tikzpicture}[scale=.5mm]
\draw[gray,fill, line width=2] (0,1) -- (-1,0)  -- (0,-1) -- (0,1);
\draw[step=1cm,black,very thin] (-2,-2) grid (2,2);
\filldraw[black] (0,0) circle (2pt);
\draw[black] (1,2) circle (2pt);
\draw[black, dotted, line width=2] (-.5,-2) -- (1.5,2);
\draw[black, dashed, line width=1,gray] (1,2) -- (0,-2);
\draw (1,0) node[cross=4pt,red] {};
\end{tikzpicture}
\caption{Case \ref{casetres} in proof of Theorem \ref{thm:3dbound}}\label{fig:casetres}
\end{figure}
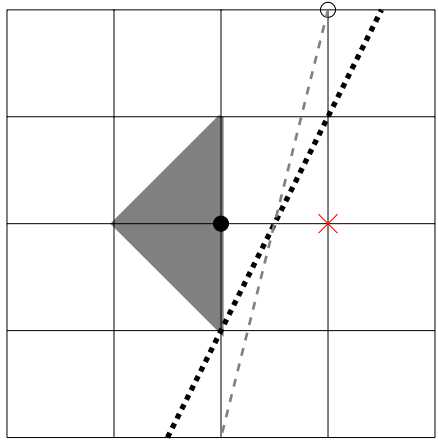

The proof of Theorem~\ref{thm:3dbound} follows the spirit of the above
examples. 

\begin{proof}[Proof of Theorem \ref{thm:3dbound}]
Let $\Psi$ consist of functions $\Psi_1,\ldots,\Psi_n$, which we may
assume to be non-linear or non-integral. There are three main
possibilities for the base polytope $\Box$.
  \begin{case}\label{caseuno}
$\Box$ contains a cross polytope. 
\end{case}
\noindent In this case, Theorem \ref{bound} directly implies that
$n\leq 8$. 

If, rather,  $\Box$ does not contain a cross polytope, Lemma
\ref{lemma:orth} implies that, we may choose a basis $e_1,e_2$ of $M$
such that $-e_1,-e_2\in \Box$, as well as some lattice point in the
interior of the positive orthant. Fix such a basis. There are only two
possibilities.
\begin{case}\label{caseduo}
Neither $e_1$ nor $e_2$ is in $\Box$.
\end{case}
\noindent Figure \ref{fig:caseduo} depicts this case, where the gray
region is contained in $\Box$, but the marked crosses in the figure
are not. Since $\Box$ contains a point in the interior of the positive
orthant, in fact the point $e_1+e_2$ must be in $\Box$. But then $\pm
e_1/2\in \Box$, and the line spanned by $e_1$ intersects the boundary
of $\Box$ in a non-lattice point, so Lemma \ref{lemma:lbound} implies
that $n\leq 8$ as desired.

\begin{case}\label{casetres}
The point $e_2$ is in $\Box$, but $e_1$ is not.
\end{case}
\noindent Since the argument of Case \ref{caseduo} applies if $e_1+e_2\in \Box$, we may thus assume that every $(x_1,x_2)\in \Box$ located in the positive orthant satisfies $x_2>2x_1-1$, see the black dotted line in Figure \ref{fig:casetres}.	
If $e_2$ is on the boundary of $\Box$, then the only lattice point of $\Box$ in the positive orthant must be $e_1+2e_2$. But then every lattice point $(x_1,x_2)$ of $\Box$ must satisfy $x_2> 4x_1-2$, otherwise $e_1/2\in \Box$ and again the argument of Case \ref{caseduo} applies, see the dashed gray line in Figure \ref{fig:casetres}.
It follows that either $\Box$ is the convex hull of $e_1+2e_2$, $-e_1$, and $-e_2$, or, that $\Box$ contains the gray region pictured in Figure \ref{fig:ep}. In the former situation, $\Box$ is equivalent to the polytope of Example \ref{ex:tri}, guaranteeing $n\leq 8$. In the latter situation, Example \ref{ex:ep} also guarantees $n\leq 8$.
	
We may thus continue with Case \ref{casetres}, assuming that $e_2$ is in the interior of $\Box$.
Let $R$ be the set of those $i$ such that $\Psi'_i$ is non-integral, or non-linear along the $e_2$-axis, with $r=|R|$. By Lemma \ref{lemma:bound}, we have $r\leq 4$. 

For any $\Psi'_i$ with $i\notin R$, the concavity of $\Psi'_i$ and linearity along the $e_2$-axis implies that $\Psi'_i$ is has exactly two domains of linearity, divided by the $e_2$-axis. For the moment, assume that $R\neq \emptyset$. For all $i\notin R$ we normalize $\Psi$ so that $\Psi'_i\equiv 1$ on the region of those points $(x_1,x_2)$ with  $x_1\leq 0$.
Consider $y=\lambda e_2$ for $\lambda$ maximal such that $y\in \Box$. By our assumptions above, $\lambda>1$, which implies that it lies on a facet of $\Box$ which is not at height one. We thus have that $n-2=\sum \Psi'_i(y)=(n-r)+ \sum_{i\in R} \Psi'_i(y)$, or equivalently that
\begin{equation}\label{eqn:y}
\sum_{i\in R} \Psi'_i(y)=r-2.
\end{equation}

Likewise, consider some $z\in \Box$ with first coordinate equal to $1$. Our normalization implies that for $i\notin R$, $\Psi'_i(z)\leq 0$. 
Set $z'=\frac{z-e_1}{2}$. This can be written as $sy+(1-s)0$ for some $s\in(0,1)$. 
See Figure \ref{fig:casetresa} for a depiction of $y$, $z$, and $z'$.
\begin{figure}
	\begin{tikzpicture}[scale=.5mm]
\filldraw[black] (0,0) circle (2pt);
\draw[black] (-1,0) circle (1pt);
\draw[black] (0,1.25) circle (1pt);
\draw[black] (0,1) circle (1pt);
\draw[black] (0,2.15) circle (1pt);
\draw[black] (1,2.5) circle (1pt);
\draw[black, dotted, line width=2] (0,-1) -- (2,3);
\draw[gray, line width=1] (-1,0) -- (1,2.5);
\draw[gray, line width=1] (0,0) -- (0,2.15);
\node [above left] at (-1,0)  {$-e_1$};
\node [right] at (1,2.5)  {$z$};
\node [right] at (0,1.25)  {$z'$};
\node [above left] at (0,2.15)  {$y$};
\node [below right] at (0,1)  {$e_2$};

\end{tikzpicture}
\caption{Case \ref{casetres} with $e_2\in\Box^\circ$ in proof of Theorem \ref{thm:3dbound}}\label{fig:casetresa}
\end{figure}
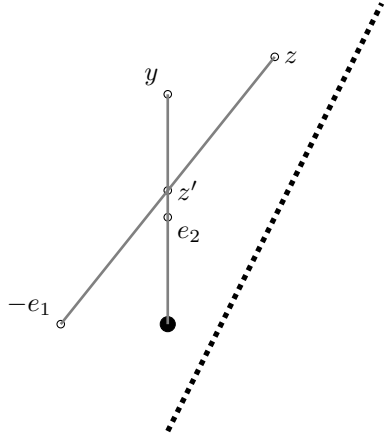

For $i\in R$, $\Psi_i'$ restricted to the $e_2$ axis is bounded above by the linear function taking values $1$ at the origin, and $\Psi'_i(y)$ at $y$.
Indeed, if $\Psi'_i$ is integral, this follows from Lemma \ref{linearalonglines}. If $\Psi'_i$ is not integral, the fact that each facet of the graph of $\Psi'_i$ (and hence its restriction to the $e_2$-axis) is at height one implies the bound on the domain of linearity closest to $y$; concavity implies the bound holds for the entire axis.

We thus obtain
\begin{align*}
	2n-4\leq \sum \Psi'_i(z)+\sum \Psi'_i(e_1)\leq \sum_{i\notin R} 1+2\sum_{i\in R} \Psi'_i(z')\\
	\leq (n-r)+2\sum_{i\in R} s\Psi'_i(y)+(1-s)=n+r-4s,
\end{align*}
where the first inequality follows from Equation \eqref{eq2}, the second from the concavity of $\Psi_i'$, the third from the bound on $\Psi_i'$ along the $e_2$ axis, and the final equality follows from Equation \eqref{eqn:y}. This in turn implies
\[
n\leq 4+r-4s. 
\]
Using that $r\leq 4$ and $s>0$, we obtain $n\leq 7$.

If instead $R=\emptyset$, then we can only normalize all but one of the functions $\Psi_i'$ not in $R$; say that $\Psi_1'$ is the function we don't normalize. Then as in Equation \eqref{eqn:y} we obtain $\Psi_i'(y)=-1$ with $y$ as above, and arguing as above (with $R$ replaced by $\{1\}$) we obtain
$2n-4\leq (n-1)-2s+2(1-s)=n+1-4s$. Hence, $n\leq4$.

We have shown that in all cases, $n\leq 8$. On the other hand, Example \ref{ex:cross} shows that this bound is sharp, as the two-dimensional cross polytope supports a Fano CDP with $8$ non-linear functions.
\end{proof}

Based on our results for two- and three-dimensional Fano CDPs, we conjecture the following:
\begin{conj}\label{conj:bound}
Any normalized Fano CDP of dimension $d$ has at most $2^d$ functions.
\end{conj}
\noindent As Example \ref{ex:cross} shows, there is a Fano CDP supported on the $(d-1)$-dimensional cross polytope which achieves this conjectural bound.

\subsection*{Acknowledgements} 
We thank Hendrik S\"u\ss{} for useful discussions. All three authors
were supported by the NSERC Discovery Grant Program, and CT was
also supported by an NSERC USRA Fellowship.

\bibliographystyle{alpha}

\bibliography{cdp} 

\begin{thebibliography}{BHHN16}

\bibitem[BHHN16]{bechtold:16a}
Benjamin Bechtold, J\"urgen Hausen, Elaine Huggenberger, and Michele Nicolussi.
\newblock On terminal {F}ano 3-folds with 2-torus action.
\newblock {\em Int. Math. Res. Not. IMRN}, (5):1563--1602, 2016.

\bibitem[CLS11]{CLS}
David~A. Cox, John~B. Little, and Henry~K. Schenck.
\newblock {\em Toric varieties}, volume 124 of {\em Graduate Studies in
  Mathematics}.
\newblock American Mathematical Society, Providence, RI, 2011.

\bibitem[CS17]{hendrik}
Jacob Cable and Hendrik S\"u\ss{}.
\newblock On the classification of {K}\"ahler-{R}icci solitions on {G}orenstein
  del {P}ezzo surfaces.
\newblock arXiv:1705.02920v1, 2017.

\bibitem[FHN17]{pic2}
Anne Fahrner, J\"urgen Hausen, and Michele Nicolussi.
\newblock Smooth projective varieties with a torus action of complexity 1 and
  {P}icard number 2.
\newblock arXiv:1602.04360v3, 2017.

\bibitem[Ful93]{fulton}
William Fulton.
\newblock {\em Introduction to toric varieties}, volume 131 of {\em Annals of
  Mathematics Studies}.
\newblock Princeton University Press, Princeton, NJ, 1993.
\newblock The William H. Roever Lectures in Geometry.

\bibitem[Hau13]{hausen:13a}
J\"urgen Hausen.
\newblock Three lectures on {C}ox rings.
\newblock In {\em Torsors, \'etale homotopy and applications to rational
  points}, volume 405 of {\em London Math. Soc. Lecture Note Ser.}, pages
  3--60. Cambridge Univ. Press, Cambridge, 2013.

\bibitem[HHS11]{hausen:11a}
J\"urgen Hausen, Elaine Herppich, and Hendrik S\"uss.
\newblock Multigraded factorial rings and {F}ano varieties with torus action.
\newblock {\em Doc. Math.}, 16:71--109, 2011.

\bibitem[Hug13]{huggenberger}
Elaine Huggenberger.
\newblock {\em Fano varieties with torus action of complexity one}.
\newblock PhD thesis, Doctoral Thesis http://nbn-resolving. de/urn: nbn: de:
  bsz: 21-opus-69570, 2013.

\bibitem[IS11]{polarized}
Nathan~Owen Ilten and Hendrik S\"uss.
\newblock Polarized complexity-1 {$T$}-varieties.
\newblock {\em Michigan Math. J.}, 60(3):561--578, 2011.

\bibitem[IS17]{stability}
Nathan Ilten and Hendrik S\"uss.
\newblock K-stability for {F}ano manifolds with torus action of complexity 1.
\newblock {\em Duke Math. J.}, 166(1):177--204, 2017.

\bibitem[KS97]{kreuzer:97a}
M.~{Kreuzer} and H.~{Skarke}.
\newblock {On the Classification of Reflexive Polyhedra}.
\newblock {\em Communications in Mathematical Physics}, 185:495--508, 1997.

\bibitem[LZ91]{lagarias:91a}
Jeffrey~C. Lagarias and G\"unter~M. Ziegler.
\newblock Bounds for lattice polytopes containing a fixed number of interior
  points in a sublattice.
\newblock {\em Canad. J. Math.}, 43(5):1022--1035, 1991.

\end{thebibliography}

\end{document}